\documentclass[12pt,reqno]{amsart}
\usepackage{amssymb}
\usepackage{amscd}
\usepackage{amsmath}

\newfont{\cyr}{wncyr10} %this gives us cyrillic fonts

\theoremstyle{plain}
\newtheorem{theorem}{Theorem}[section]
\newtheorem{Thm}{Theorem}
\newtheorem{lemma}[theorem]{Lemma}
\newtheorem{proposition}[theorem]{Proposition}
\newtheorem{corollary}[theorem]{Corollary}

\newtheorem{question}[theorem]{Question} 

\newtheorem*{proposition*}{Proposition}

\theoremstyle{definition}
\newtheorem{remark}[theorem]{Remark}
\newtheorem{definition}[theorem]{Definition}
\newtheorem{example}[theorem]{Example}

\numberwithin{equation}{section}

\newcommand{\bA}{{\mathbf A}}
\newcommand{\fA}{{\mathfrak A}}
\newcommand{\fa}{{\mathfrak a}}

\newcommand{\fb}{{\mathfrak b}}
\newcommand{\C}{{\mathcal C}}
\newcommand{\cC}{{\mathcal C}}
\newcommand{\bC}{{\mathbf C}}
\newcommand{\fc}{{\mathfrak c}}

\newcommand{\cD}{{\mathcal D}}

\newcommand{\bF}{{\mathbf F}}

\newcommand{\ff}{{\mathfrak f}}
\newcommand{\fF}{{\mathfrak F}}
\newcommand{\hG}{{\widehat G}}

\newcommand{\cH}{{\mathcal{H}}}

\newcommand{\1}{{\mathbf 1}}  % `boldface one'.

\newcommand{\cM}{{\mathcal M}}

\newcommand{\cP}{{\mathcal P}}
\newcommand{\fp}{{\mathfrak p}}
\newcommand{\fP}{{\mathfrak P}}

\newcommand{\bQ}{{\mathbf Q}}

\newcommand{\cR}{{\mathcal R}}

\newcommand{\br}{{\mathbf r}}
\newcommand{\cW}{{\mathcal W}}

\newcommand{\bZ}{{\mathbf Z}}

\DeclareMathOperator{\Cl}{Cl}
\DeclareMathOperator{\co}{co}

\DeclareMathOperator{\disc}{disc}

\DeclareMathOperator{\full}{full}
\DeclareMathOperator{\Gal}{Gal}

\DeclareMathOperator{\Hom}{Hom}
\DeclareMathOperator{\Image}{Im}
\DeclareMathOperator{\Ker}{Ker}

\DeclareMathOperator{\Map}{Map}

\DeclareMathOperator{\rag}{rag}
\DeclareMathOperator{\ram}{ram}

\begin{document}
\title[On counting rings of integers]{On counting rings of
  integers as Galois modules}
\author{A. Agboola}
\date{Final version. August 16, 2010.}
\address{Department of Mathematics \\
University of California \\ Santa Barbara, CA 93106. }
\email{agboola@math.ucsb.edu}
%\subjclass[2000]{11G05, 11R23, 11G16}
\thanks{Partially supported by NSF and NSA grants.}
%\keywords{Rubin, $p$-adic $L$-function, Birch and Swinnerton-Dyer
%  conjecture, elliptic curve, Mazur-Tate-Teitelbaum.}
\begin{abstract}
Let $K$ be a number field and $G$ a finite abelian group. We study the
asymptotic behaviour of the number of tamely ramified $G$-extensions
of $K$ with ring of integers of fixed realisable class as a Galois
module. 
\end{abstract}

\maketitle
%%%%%%%%%%%%%%%%%%%%%%%%%%%%%%%%%%%%%%%%%%%%%%%%%%%%%%%%%%%%%%%%%%%

\section{Introduction}

Suppose that $K$ is a number field with ring of integers $O_K$, and
let $G$ be a fixed, finite group. If $K_h/K$ is a tamely ramified
Galois algebra with Galois group $G$, then a classical theorem of
E. Noether implies that the ring of integers $O_h$ of $K_h$ is a
locally free $O_KG$-module. It therefore determines a class $(O_h)$ in
the locally free class group $\Cl(O_KG)$ of $O_KG$. We say that a
class $c \in \Cl(O_KG)$ is \textit{realisable} if $c = (O_h)$ for some
tamely ramified $G$-algebra $K_h/K$, and we write $\cR(O_KG)$ for the
set of realisable classes in $\Cl(O_KG)$. These classes are natural
objects of study, and they arise, for instance, in the context of
obtaining explicit analogues of known Adams-Riemann-Roch theorems for
locally free class groups (see e.g. \cite[\S 4]{AB} and the references
cited there; see also the work of B. K\"ock (\cite{K1}, \cite{K2}) on
this and related topics). We also remark that the problem of
describing $\cR(O_KG)$ for arbitrary finite groups $G$ may be viewed
as being a Galois module theoretic analogue of the inverse Galois
problem for finite groups.

When $G$ is abelian, Leon McCulloh has obtained a complete description
of $\cR(O_KG)$ in terms of certain Stickelberger homomorphisms on
classgroups (see \cite{Mc}). In particular, he has shown that
$\cR(O_KG)$ is in fact a group. Suppose now that $c \in \cR(O_KG)$,
and write $N_{\disc}(c,X)$ for the number of tame $G$-extensions
$K_h/K$ for which $(O_h)=c$ and $\disc(K_h/\bQ) \leq X$, where
$\disc(K_h/\bQ)$ denotes the absolute value of the discriminant of
$K_h/\bQ$. The following very natural counting problem appears to have
received surprisingly little attention.

\begin{question} 
What can be said about $N_{\disc}(c,X)$ as $X \to \infty$?  For example,
if $M_{\disc}(X)$ denotes the number of tame $G$-extensions $K_h/K$ for
which $\disc(K_h/\bQ) \leq X$, is 
\begin{equation*}
\lim_{X \to \infty}
\frac{N_{\disc}(c,X)}{M_{\disc}(X)}
\end{equation*} 
independent of the realiasable class $c$?
\end{question}

The only previous results concerning this question of which the author
is aware are those contained in the unpublished University of Illinois
Ph.D. thesis of Kurt Foster (see \cite{Fo}). Foster considers the case
in which $G$ is an elementary abelian $l$-group for some prime
$l$. Using earlier work of McCulloh on realisable classes for
elementary abelian groups (see \cite{Mc1}), he proves the following
result.

\begin{Thm}(K. Foster) \label{T:foster}
Suppose that $G$ is an elementary abelian
  $l$-group. Then 
\begin{equation*}
N_{\disc}(c,X) \sim \beta \cdot Y \cdot(\log Y)^{r-1} 
\end{equation*}
as $X \to \infty$, where

$\bullet$ $Y^{\phi(|G|)} (\disc(K/\bQ))^{|G|} = X$ (here $\phi$ denotes
  the Euler $\phi$-function);

$\bullet$ $\beta$ is a positive constant that depends upon $K$ and
  $G$, but not on $c$;

$\bullet$ $r$ is a positive integer that depends only upon $K$ and $G$.
\end{Thm}

Hence, when $G$ is an elementary abelian group, then asymptotically
$N_{\disc}(c,X)$ is independent of $c$, and so we see that the tame
$G$-extensions of $K$ are equidistributed amongst the realisable
classes as $X \to \infty$.

Let us say a few words about the main ideas involved in the proof of
Theorem \ref{T:foster}. One begins by considering the series
\begin{equation} \label{E:dfoster}
\sum_{ \substack{ K_{h}/K\, \text{tame}, \\
\Gal(K_{h}/K) \simeq G \\
(O_{h}) = c} }
\disc(K_h/\bQ)^{-s}, \qquad s \in \bC.
\end{equation}
Of course it is not a priori clear that this series converges
anywhere; one establishes convergence in some right-hand half-plane by
showing that it may be written as an Euler product over rational
primes. The series may therefore be written in the form
$\sum_{n=1}^{\infty} a_n n^{-s}$. One deduces from this that in
general, the series will have finitely many poles (whose locations may
be determined), and that the number $N_{\disc}(c,X)$ is equal to
$\sum_{n \leq X} a_n$. This last quantity may then be estimated by
using a suitable version of the D\'elange-Ikehara Tauberian theorem.

Our goal in this paper is to investigate similar counting problems
when $G$ is an arbitrary finite abelian group. We shall do this by
combining Foster's approach with later work of McCulloh (see
\cite{Mc}) on realisable classes for arbitrary finite abelian groups.

A special case of our main result (see Theorem \ref{T:main}) may be
described as follows. Let $G$ be an arbitrary finite abelian
group. For any tame $G$-extension $K_h/K$, let $\cD(K_h/K)$ denote the
absolute norm of the product of the primes of $K$ that ramify in
$K_h/K$. If $c \in \cR(O_KG)$, then we write $N_{\cD}(c,X)$ for the
number of tame $G$-extensions $K_h/K$ such that $(O_h) = c$,
$\cD(K_h/K) \leq X$, and $K_h/K$ is unramified at all places dividing
$|G|$. The following result shows that asymptotically, $N_{\cD}(c,X)$
is independent of $c$.

\begin{Thm} \label{T:B}
With notation and hypotheses as above, we have
\begin{equation*}
N_{\cD}(c,X) \sim \beta_1 \cdot X \cdot (\log X)^{r_1-1},
\end{equation*}
as $X \to \infty$. Here $\beta_1$ is a constant depending only upon
$K$ and $G$, but not upon $c$, and $r_1$ is a positive integer that
depends only upon $K$ and $G$. 
%is the same positive integer
%that occurs in the statement of Theorem \ref{T:foster}.
\end{Thm}

For arbitrary finite abelian $G$, our results concerning
$N_{\disc}(c,X)$ are unfortunately not as precise (see
\eqref{E:countfunc} and Section \ref{S:sigh}). The results that we
obtain indicate that it is very unlikely that the analogue of Foster's
equidistribution result holds in general, although at present we are
unable to prove this. This fact, namely that when tame $G$-extensions
of $K$ are counted by discriminant, then in general, they are probably
not equidistributed amongst the realisable classes, was rather
surprising to us. It is interesting to compare the results of this
paper with recent work of Melanie Wood on a quite different type of
counting problem (see \cite{Wood}). Wood studies the probabilities of
various local completions of a random $G$-extension of $K$. She proves
that these probabilities are well-behaved and are--for the most
part--independent when $G$-extensions of $K$ are counted by conductor;
as she points out, this is in close analogy with Chebotarev's density
theorem. When $G$-extensions of $K$ are counted by discriminant
however, she proves that these probabilities are poorly behaved and in
general are not independent. It would be interesting to obtain a
better understanding of the relationship, if any, between the results
described in the present paper and those of \cite{Wood}.

An outline of the contents of this paper is as follows. In Section
\ref{S:leon} we review McCulloh's theory of realisable classes. In
Section \ref{S:count}, we use the methods of \cite{Fo} to set up a
counting problem that will enable us to analyse the distribution of
tame $G$-extensions of $K$ amongst realisable classes. In Sections
\ref{S:euler} and \ref{S:asymptotics} we study analogues of the series
\eqref{E:dfoster} in our setting. We show that they are Euler
products, and we apply a Tauberian theorem in order to state a result
concerning their asymptotic behaviour. In Section \ref{S:dirichlet} we
introduce certain Dirichlet $L$-series; these are then used in Section
\ref{S:poles} to determine the location of the poles of the series
introduced in Section \ref{S:euler}. We state our main result in
Section \ref{S:equi}, and we explain how it may be used to recover
Theorem \ref{T:foster} and to prove Theorem \ref{T:B}. In Section
\ref{S:field}, we discuss why, in many cases, it makes no difference
if we count tamely ramified Galois field extensions of $K$ with Galois
group $G$, rather that tamely ramified $G$-algebra extensions of
$K$. Finally, in Section \ref{S:sigh} we explain why our results
indicate that the analogue of Foster's equidistribution result
probably does not hold in general, and we discuss what would need to
be done to produce an explicit counterexample.
\medskip

\noindent{}{\bf Acknowledgements.} It will be clear to the reader that
this paper owes a great deal to the work of L. McCulloh and
K. Foster. I am very grateful to Leon McCulloh for sending me a copy
of Foster's thesis. I would also like to thank Jordan Ellenberg for
his interest, Melanie Wood for sending me a copy of her paper
\cite{Wood}, and the anonymous referee for many
extremely helpful comments.
\medskip

\noindent{}{\bf Notation and conventions.} 
If $L$ is a number field, we write $O_L$ for its ring of integers. We
set $\Omega_L:= \Gal(L^c/L)$, where $L^c$ denotes an algebraic closure
of $L$, and we write $I(O_L)$ for the group of fractional ideals of
$L$.
\smallskip

The symbol $G$ will always denote a finite, abelian group. If $H$ is
any group, we write $\widehat{H}$ for the group of characters of $H$,
and $\1_H$ (or simply $\1$ if there is no danger of confusion) for the
trivial character in $\widehat{H}$.
\smallskip

We identify $G$-Galois algebras of $K$ with elements of $H^1(K,G)
\simeq \Hom(\Omega_K, G)$ (see \ref{SS:resolvends} below). If $h \in
H^1(K,G)$, then we write $K_h/K$ for the corresponding $G$-extension
of $K$, and $O_h$ for the integral closure of $O_K$ in $O_h$. We write
$H^{1}_{tr}(K,G)$ for the subgroup of $H^1(K,G)$ consisting of those
$h \in H^1(K,G)$ for which $K_h/K$ is tamely ramified.
\smallskip

If $L/K$ is any finite extension, then $N_{L/K}$ denotes the norm map
from $L$ to $K$.
\medskip

%%%%%%%%%%%%%%%%%%%%%%%%%%%%%%%%%%%%%%%%%%%%%%%%%%%%%%%%%%%%%%%%%%%%%%%%%%
\section{Review of McCulloh's theory of realisable classes}
\label{S:leon}

In this section we shall briefly describe McCulloh's theory of
realisable classes of tame extensions. The reader is strongly
encouraged to consult McCulloh's paper \cite{Mc} for full details. 

\subsection{Locally free class groups} \cite[Section 3]{Mc}. 
In this subsection we shall recall some basic facts concerning the
Picard group $\Cl(O_KG)$ of $O_KG$.

Let $J(KG)$ denote the group of finite ideles of $KG$, i.e. the
restricted direct product of the groups $(K_vG)^{\times}$ with respect
to the subgroups $(O_{K,v}G)^{\times}$. Then there is a natural
isomorphism
\begin{equation} \label{E:isoclass}
\Cl(O_KG) \simeq \frac{J(KG)}{\left( \prod_v(O_{K,v}G)^{\times}
  \right) \left(KG\right)^{\times}}.
\end{equation}
Suppose that $K_{h}/K$ is a tamely ramified Galois algebra with
$\Gal(K_{h}/K) \simeq G$. Then by Noether's theorem, the ring of
integers $O_{h}$ of $K_{h}$ is a locally free $O_KG$-module of
rank one. Let $b \in K_{h}$ be a $KG$-generator of $K_{h}$, and,
for each finite place $v$ of $K$, choose an $O_{K,v}G$-generator $a_v$
of $O_{h,v}$. We refer to $b$ as a \textit{normal basis generator}
and to $a_v$ as a \textit{normal integral basis generator}. Then there
exists $c_v \in (K_vG)^{\times}$ such that $a_v = c_vb$. It may be
shown that $c=(c_v)_v \in J(KG)$. The idele $c$ is a representative of
$(O_{h}) \in \Cl(O_KG)$.

Now let
\begin{equation*}
j: J(KG) \to \Cl(O_KG)
\end{equation*}
denote the surjective homomorphism afforded by the isomorphism
\eqref{E:isoclass}, and suppose that $c$ is any idele in
$J(KG)$. How can we tell whether or not the class $j(c)$ is
realisable? In order to describe the answer to this question, we need
to introduce some further ideas and notation.

\subsection{Resolvends} \label{SS:resolvends} \cite[Section 1]{Mc}.
If $h: \Omega_K \to G$ is any continuous homomorphism, then we may
define an associated $G$-Galois $K$-algebra $K_h$ by
\begin{equation*}
K_h:= \Map_{\Omega_K}({^hG}, K^c),
\end{equation*}
where ${^hG}$ denotes the set $G$ endowed with an action of
$\Omega_K$ via the homomorphism $h$, and $K_h$ is the algebra of
$K^c$-valued functions on $G$ that are fixed under the action of
$\Omega_K$. The group $G$ acts on $K_h$ via the rule
\begin{equation*}
a^s(t) = a(ts)
\end{equation*}
for all $s, t \in G$. It may be shown that every $G$-Galois
$K$-algebra is isomorphic to an algebra of the form $K_h$ for some
$h$. Every $G$-Galois $K$-algebra may therefore be viewed as lying in
the $K^c$-algebra $\Map(G,K^c)$. It is therefore natural to consider
the Fourier transforms of elements of $\Map(G,K^c)$. These arise via
the \textit{resolvend} map
\begin{equation*}
\br_G: \Map(G,K^c) \to K^cG; \qquad a \mapsto \sum_{s \in G} a(s)
s^{-1}.
\end{equation*}
The map $\br_G$ is an isomorphism of left $K^cG$-modules, but not of
algebras, because it does not preserve multiplication. It is not hard
to show that for any $a \in \Map(G, K^c)$, we have that $a \in K_h$ if
and only if $\br_G(a)^{\omega} = \br_G(a)h(\omega)$ for all $\omega
\in \Omega_K$ (where here $\Omega_K$ acts on $K^cG$ via its
action on the coefficients). It may also be shown that an element $a
\in K_h$ generates $K_h$ as a $KG$-module if and only if $\br_G(a) \in
(K^cG)^{\times}$. Two elements $a_1, a_2 \in \Map(G,K^c)$ with
$\br_G(a_1), \br_G(a_2) \in (K^cG)^{\times}$ generate the same
$G$-Galois $K$-algebra as a $KG$-module if and only if $\br_G(a_1) =
b \cdot \br_G(a_2)$ for some $b \in (K^cG)^{\times}$.

We define
\begin{align*}
&H(KG):= \left\{ \alpha \in (K^cG)^{\times}:
  \alpha^{\omega}/\alpha \in G \quad \forall \omega \in \Omega_K
  \right\}; \\
&\cH(KG):= H(KG)/G.
\end{align*}
The group $H(KG)$ consists precisely of resolvends of normal basis
generators of $G$-Galois $K$-algebras lying in $\Map(G, K^c)$. The
group $\cH(KG)$ may be naturally identified with the set of all normal
basis generators of $G$-Galois $K$-algebras lying in $\Map(G,K^c)$.

For each finite place $v$ of $K$, we define $\cH(K_vG)$ and
$\cH(O_{K,v}G)$ analogously. We write $\cH(\bA(KG))$ for the restricted
direct product of the groups $\cH(K_vG)$ with respect to the groups
$\cH(O_{K,v}G)$. Then the natural maps
\begin{equation*}
(K_vG)^{\times} \to \cH(K_vG)
\end{equation*}
induce a homomorphism
\begin{equation*}
\rag: J(KG) \to \cH(\bA(KG)).
\end{equation*}

McCulloh shows that if $c \in J(KG)$, then $j(c) \in \Cl(O_KG)$ is
realisable if and only if $\rag(c)$ admits a certain local
decomposition. This local decomposition involves certain Stickelberger
maps that we shall now describe.

\subsection{Stickelberger maps} \cite[Section 4]{Mc}.
Let $\hG$ denote the group of complex-valued characters of $G$, and
write $G(-1)$ for the group $G$ endowed with a $\Omega_K$-action
via the inverse cyclotomic character. There is a natural pairing
\begin{equation*}
\langle \, ,\,\rangle: \bQ \hG \times \bQ G \to \bQ
\end{equation*}
defined by 
\begin{equation*}
\chi(g) = \exp(2 \pi i \langle \chi,g \rangle), \qquad 0 \leq \langle
\chi, g \rangle < 1
\end{equation*}
for $\chi \in \hG$ and $g \in G$. This pairing may in turn be used to
define a \textit{Stickelberger map}
\begin{equation*}
\Theta: \bQ \hG \to \bQ G; \qquad \alpha \mapsto \sum_{g \in G}
\langle \alpha, g \rangle g.
\end{equation*}

Let $A_{\hG}$ denote the kernel of the determinant map
\begin{equation*}
\det: \bZ \hG \to \hG; \qquad \sum_{\chi \in \hG} a_{\chi} \chi
\mapsto \prod_{\chi \in \hG} \chi^{a_{\chi}}.
\end{equation*}
Then the standard isomorphism
\begin{equation*}
(K^cG)^{\times} \simeq \Hom(\bZ \hG, (K^c)^{\times})
\end{equation*}
induces an isomorphism
\begin{equation*}
(K^cG)^{\times}/G \simeq \Hom(A_{\hG}, (K^c)^{\times}).
\end{equation*}

\begin{proposition} \label{P:AG} 
(McCulloh) If $\alpha \in \bZ \hG$,
then $\Theta(\alpha) \in \bZ G$ if and only if $\alpha \in A_{\hG}$.
\end{proposition}

\begin{proof} 
See \cite[Proposition 4.3]{Mc}.
\end{proof}

Proposition \ref{P:AG} implies that, via restriction, $\Theta$ defines
a homomorphism (which we denote by the same symbol)
\begin{equation*}
\Theta: A_{\hG} \to \bZ G.
\end{equation*}
Dualising this homomorphism, and twisting by the inverse cyclotomic
character yields an $\Omega_K$-equivariant \textit{transpose
Stickelberger homomorphism}
\begin{equation*}
\Theta^t: \Hom (\bZ G(-1), (K^c))^{\times}) \to \Hom(A_{\hG},
(K^c)^{\times}) \simeq (K^cG)^{\times}/G.
\end{equation*}

Now set
\begin{align*}
&\Lambda:= \Hom_{\Omega_K}(\bZ G(-1), O_{K^c}) =
  \Map_{\Omega_K}(G(-1),O_{K^c}); \\
& K\Lambda:= \Hom_{\Omega_K}(\bZ G(-1), K^c) =
  \Map_{\Omega_K}(G(-1),K^c).
\end{align*}
Then $\Theta^t$ above induces a homomorphism
\begin{equation*}
\Theta^t: (K\Lambda)^{\times} \to [(K^cG)^{\times}/G]^{\Omega_K} =
\cH(KG).
\end{equation*}
For each finite place $v$ of $K$, we can apply the discussion above
with $K$ replaced by $K_v$ to obtain a local version
\begin{equation} \label{E:locstick}
\Theta_{v}^{t}: (K_v\Lambda_v)^{\times} \to \cH(K_vG)
\end{equation}
of the map $\Theta^t$. The homomorphism $\Theta^t$ commutes with local
completion.

For all places $v$ of $K$ not dividing the order of $G$, it may be
shown that $\Theta^t(\Lambda_v) \subseteq \cH(O_{K,v}G)$. Hence if we
write $J(K\Lambda)$ for the restricted direct product of the groups
$(K_v\Lambda_v)^{\times}$ with respect to the groups
$\Lambda_{v}^{\times}$, then the homomorphisms $\Theta^{t}_{v}$ combine to
yield an idelic transpose Stickelberger homomorphism
\begin{equation} \label{E:idstick}
\Theta^t: J(K\Lambda) \to \cH(\bA(KG)).
\end{equation}

\subsection{Prime $F$-elements} \cite[Section 5]{Mc} 
Let $v$ be a finite place of $K$, and write $q_v$ for the order of the
residue field at $v$. Fix a local uniformiser $\pi_v$ of $K$ at
$v$. Write $G_{(q_v-1)}$ for the subgroup of $G$ consisting of all
elements in $G$ of order dividing $q_v-1$.

For each element $s \in G_{(q_v-1)}$, define $f_{v,s} \in
(K_v\Lambda_v)^{\times} = \Map(G(-1),
(K_{v}^{c})^{\times})^{\Omega_K}$ by
\begin{equation}
f_{v,s}(t) =
\begin{cases}
\pi_v, &\text{if $t=s \neq 1$;} \\ 1, &\text{otherwise.}
\end{cases}
\end{equation}
Note in particular that $f_{v,1} = 1$.

Write
\begin{equation*}
\bF_v:= \{f_{v,s}\, \mid \, s \in G_{(q_v-1)}\}.
\end{equation*}
The non-trivial elements of $\bF_v$ are called \textit{the prime
  $F$-elements lying above $v$}. We define $\bF \subset J(K\Lambda)$ by
\begin{equation*}
f \in \bF \iff \text{$f \in J(K\Lambda)$ and $f_v \in \bF_v$ for
all $v$.}
\end{equation*}
In other words, each non-trivial element of $\bF$ is a finite product of
prime $F$-elements lying over distinct places $v$ of $K$; in
particular, if $f \in \bF$, then $f_v =1$ for almost all $v$.
\smallskip

We can now state two results of McCulloh. The first result (see
\cite[Theorem 6.7]{Mc}) characterises tame $G$-extensions of $K$ in
terms of resolvends of normal basis generators. The second (see
\cite[Theorem 6.17]{Mc}) gives a precise characterisation of those
ideles $c \in J(KG)$ for which $j(c) \in \Cl(O_KG)$ is realisable.

Set
\begin{equation*}
\cH(\bA(O_KG)):= \prod_v \cH(O_{K,v}G).
\end{equation*}

\begin{theorem} \label{T:leon1}(McCulloh)
Suppose that $c \in J(KG)$. Then $j(c) = (O_h)$ for some tamely
ramified $G$-Galois algebra extension $K_h/K$ (i.e. $j(c)$ is
realisable) if and only if there exist $b \in \cH(KG)$, $f \in \bF$
and $u \in \cH(\bA(O_KG))$ such that
\begin{equation*}
\rag(c) = b^{-1}\cdot \Theta^t(f) \cdot u \in \cH(\bA(KG)).
\end{equation*}
The elements $b \in \cH(KG)$ and $f \in \bF$ are uniquely
determined by $c$. Furthermore, $K_h/K$ is ramified at precisely those
places $v$ of $K$ for which $f_v \neq 1$.
\end{theorem}

\begin{theorem} \label{T:leon2}(McCulloh)
Suppose that $c \in J(KG)$. Then $j(c) \in \Cl(O_KG)$ is realisable if
and only if $\rag(c) \in \cH(KG) \cdot \cH(\bA(O_KG)) \cdot
\Theta^t(J(K\Lambda))$. 
\end{theorem}

%%%%%%%%%%%%%%%%%%%%%%%%%%%%%%%%%%%%%%%%%%%%%%%%%%%%%%%%%%%%%%%%%%%%%%%%%%%
\section{A counting problem} \label{S:count}

In this section we shall explain how to set up a counting problem that
will enable us to study the distribution of tame $G$-extensions of $K$
amongst realisable classes.  We apply a modified version of a method
described by Foster in \cite[Chapters II and III]{Fo}.

Set
\begin{equation} \label{E:C}
\cC(O_KG) := \frac{\cH(\bA(KG))}{[(KG)^{\times}/G] \cdot \cH(\bA(O_KG))}.
\end{equation}

\begin{definition} \label{D:psi}
We define a homomorphism
\begin{equation}
\psi: H^1(K,G) \to \cC(O_KG)
\end{equation}
as follows. Let $K_h/K$ be the Galois $G$-extension of $K$
corresponding to $h \in H^1(K,G)$, and let $b \in K_h$ be any normal
basis generator. We define $\psi(h)$ to be the image of $h$ under the
composition of maps
\begin{equation*}
H^1(K,G) \to \frac{H(KG)}{(KG)^{\times}} \to \cC(O_K G),
\end{equation*}
where the first arrow is given by $h \mapsto [\br_G(b)]$, and the second
arrow is induced by the diagonal embedding
\begin{equation*}
H(KG) \to \prod_v H(K_vG).
\end{equation*}
It is not hard to check that $\psi(h)$ is independent of the choice of
$b$, and that $\psi$ is a homomorphism.
\end{definition}

\begin{definition} \label{D:rho}
We define
\begin{equation*}
\rho: \Cl(O_KG) \simeq \frac{J(KG)}{(KG^{\times}) \cdot
\prod_v(O_{K,v}G)^{\times}} \to \cC(O_KG)
\end{equation*}
to be the homomorphism induced by the composition of maps
\begin{equation*}
J(KG) \to H(\bA(KG)) \to \cH(\bA(KG)).
\end{equation*}
Here the first arrow is the diagonal embedding, and the second map is the
obvious quotient homomorphism.
\end{definition}

\begin{definition} \label{D:theta}
We define
\begin{equation*}
\theta: J(K\Lambda) \to \cC(O_KG)
\end{equation*}
to be the composition
\begin{equation*}
J(K\Lambda) \xrightarrow{\Theta^t} \cH(\bA(KG)) \to \cC(O_KG),
\end{equation*}
where the second arrow denotes the natural quotient map.
\end{definition}

\begin{proposition} \label{P:maps}
(a) We have that $h \in \Ker(\psi)$ if and only if $K_h/K$ is
  unramified at all finite places of $K$ and $O_h$ is $O_KG$-free. In
  particular, $\Ker(\psi)$ is finite.

(b) The homomorphism $\rho$ is injective.

(c) The map $\theta|_{\bF}$ is injective.
\end{proposition}

\begin{proof}
(a) Suppose that $h \in \Ker(\psi)$, with $K_h =KG\cdot b$. Then it
  follows from the definition of $\psi$ that the image of $\br_G(b)$
  under the diagonal embedding $H(KG) \to \prod_v H(K_vG)$ lies in
  $(KG)^{\times} \cdot H(\bA(O_KG))$. Hence, replacing $b$ by $\alpha
  \cdot b$ for a suitably chosen element $\alpha \in (KG)^{\times}$,
  we may in fact assume that the image of $\br_G(b)$ under this
  diagonal embedding lies in $H(\bA(O_KG))$. This happens if and only
  if $K_h/K$ is unramified (see \cite[(2.12) and (2.13)]{Mc}) and
  $O_h$ is $O_KG$-free (see \cite[Theorem 5.6]{Mc}).

(b) It follows directly from the definitions of $J(KG)$ and
  $H(\bA(O_KG))$ that in $\prod_v H(K_vG)$, we have
\begin{equation*}
J(KG) \cap H(\bA(O_KG)) = \prod_v (O_{K,v}G)^{\times}.
\end{equation*}
The injectivity of $\rho$ is now a direct consequence of \eqref{E:C},
\eqref{E:isoclass}, and the definition of $\rho$.

(c) We first recall from the definition of $\bF$ that if $f = (f_v) \in \bF$, then $f_v=1$ for
almost all $v$. The proof of \cite[Proposition 5.4]{Mc} shows that for
each finite place $v$ of $K$, and $s_1, s_2 \in G_{q_v-1}$, we have
$\Theta^t(f_{v,s_{1}}) = \Theta^t(f_{v,s_{2}})$ if and only if
$s_1=s_2$. It follows that the restriction of the homomorphism
$\Theta^t$ to $\bF$ is injective.

Next we note that if $f_v \neq 1$, then plainly $f_v \notin
\Lambda^{\times}_{v}$, and it is not hard to check that $\Theta^t(f_v)
\notin H(O_{K,v}G)$. We deduce that, in $\prod_v H(K_vG)$, we have
\begin{equation*}
\Theta^t(\bF) \cap [ (KG)^{\times} \cdot H(\bA(KG))] = \{ 1\},
\end{equation*}
and this in turn implies that the restriction of the quotient map
$\cH(\bA(KG)) \to \cC(O_KG)$ to $\Theta^t(\bF)$ is injective. It
follows that the restriction of $\theta$ to $\bF$ is injective, as claimed.
\end{proof}

\begin{remark} \label{R:newleon}
(1) Suppose that $h \in H^{1}_{tr}(K,G)$. Then Theorem \ref{T:leon1}
    implies that there exists a unique $c \in \Cl(O_KG)$ (namely
    $(O_h)$) and a unique $f \in \bF$ such that
\begin{equation} \label{E:realisable}
\rho(c) = \psi(h)^{-1} \theta(f).
\end{equation}
For fixed $c \in \cR(O_KG)$ and fixed $f \in \bF$, Proposition
\ref{P:maps}(a) implies that if \eqref{E:realisable} is satisfied by
some $h \in H^{1}_{tr}(K,G)$, then in fact there are exactly $|\Ker(\psi)|$
elements $h \in H^{1}_{tr}(K,G)$ which satisfy \eqref{E:realisable}.

(2) Theorem \ref{T:leon2} implies that we have
\begin{equation*}
\rho(\cR(O_KG)) = \Image(\rho) \cap [\Image(\theta) \cdot
  \Image(\psi)].
\end{equation*}
\qed
\end{remark}

\begin{definition} \label{D:pdef}
We define
\begin{equation*}
\cP_{\theta}:= \left\{ x \in J(K\Lambda) \mid \theta(x) \in
\Image(\psi) \right\}.
\end{equation*}
\end{definition}

\begin{proposition} \label{P:pprop}
Suppose that $c \in \Cl(O_KG)$ with
\begin{equation*}
\rho(c) = \psi(h)^{-1} \theta(\lambda)
\end{equation*}
for some $h \in H^{1}(K,G)$ and $\lambda \in J(K\Lambda)$. Then,
for any $\mu \in J(K\Lambda)$, there exists $h_{\mu} \in
H^{1}(K,G)$ such that
\begin{equation*}
\rho(c) = \psi(h_{\mu})^{-1} \theta(\mu)
\end{equation*}
if and only if $\mu \in \lambda \cP_{\theta}$.

In particular, for any coset $x \cP_{\theta}$ of $\cP_{\theta}$ in
$J(K\Lambda)$, it follows that either
\begin{equation*}
\theta(x \cP_{\theta}) \subseteq \Image(\psi) \cdot \Image(\rho)
\end{equation*}
or
\begin{equation*}
\theta(x \cP_{\theta}) \cap [\Image(\psi) \cdot \Image(\rho)] =
\emptyset.
\end{equation*}
\end{proposition}

\begin{proof} 
Suppose that
\begin{equation*}
\rho(c)= \psi(h)^{-1} \theta(\lambda) = \psi(h_{\mu})^{-1} \theta(\mu),
\end{equation*}
with $h_{\mu} \in H^1(K,G)$ and $\mu \in J(KG)$. Then we have
\begin{equation*}
\theta(\lambda) \theta(\mu)^{-1} = \psi(h) \psi(h_{\mu})^{-1},
\end{equation*}
and so $\lambda \mu^{-1} \in \cP_{\theta}$, as claimed.

Conversely, if 
\begin{equation*}
\rho(c)= \psi(h)^{-1} \theta(\lambda)
\end{equation*}
and $\lambda = \mu \nu$ for some $\nu \in \cP_{\theta}$, then we have
\begin{align*}
\rho(c)&= \psi(h)^{-1} \theta(\lambda) \\
&= \psi(h)^{-1} \theta(\mu) \theta(\nu) \\
&= [\psi(h) \psi(h_{\nu})]^{-1} \theta(\mu)
\end{align*}
for some $h_{\nu} \in H^{1}(K,G)$, since $\nu \in \cP_{\theta}$.

This establishes the result. 
\end{proof}

We now observe that if $c \in \cR(O_KG)$ with
\begin{equation*}
\rho(c)= \psi(h)^{-1} \theta(\lambda)
\end{equation*}
for some $h \in H^1(K,G)$ and $\lambda \in J(K\Lambda)$, then Theorem
\ref{T:leon2} implies that in fact $h \in H^{1}_{tr}(K,G)$ (cf. also
Remark \ref{R:newleon}(1) above). We can therefore see from
Proposition \ref{P:pprop} that counting tame Galois $G$-extensions of
$K$ with a given realisable class is essentially equivalent to
counting elements in $\bF \cap \lambda \cP_{\theta}$ for a fixed coset
$\lambda \cP_{\theta}$ of $\cP_{\theta}$ in $J(K\Lambda)$. We
therefore now focus our attention on obtaining a good description of
$\bF \cap \lambda \cP_{\theta}$.

Fix a set of representatives $T$ of $\Omega_K\backslash G(-1)$, and
for each $t \in T$, let $K(t)$ be the smallest extension of $K$ such
that $\Omega_{K(t)}$ fixes $t$. Then the Wedderburn decomposition of
$K\Lambda$ is given by
\begin{equation} \label{E:wiso} 
K\Lambda = \Map_{\Omega_K}(G(-1), K^c) \simeq \prod_{t \in T}K(t),
\end{equation}
where the isomorphism is induced by evaluation on the elements of $T$.

\begin{definition} \label{D:modray} (See \cite[\S6]{Mc})
Let $\cM$ be an integral ideal of $O_K$. For each finite place $v$ of
$K$ we set $U_{\cM}(O_{K,v}^{c}) = (1+\cM O_{K,v}^{c}) \cap
(O_{K,v}^{c})^{\times}$.  We define
\begin{equation*}
U_{\cM}'(\Lambda_v) \subseteq (K_v \Lambda)^{\times} =
\Map_{\Omega_{v}}(G(-1), (K_{v}^{c})^{\times})
\end{equation*}
by
\begin{equation*}
U_{\cM}'(\Lambda_v):= \left\{ g_v \in (K_v\Lambda)^{\times} \mid
g_v(s) \in U_{\cM}(O_{K,v}^{c}) \quad \forall s \neq 1 \right\}
\end{equation*}
(with $g_v(1)$ allowed to be arbitrary).

Set
\begin{equation*}
U_{\cM}'(\Lambda):= \left( \prod_v U_{\cM}'(\Lambda_v) \right) \cap
J(K\Lambda).
\end{equation*}

The \textit{modified ray class group modulo} $\cM$ of $\Lambda$ is
defined by
\begin{equation*}
\Cl_{\cM}'(\Lambda):= \frac{J(K\Lambda)}{(K \Lambda)^{\times} \cdot
  U_{\cM}'(\Lambda)}.
\end{equation*}
The group $\Cl_{\cM}'(\Lambda)$ is finite, and is isomorphic to the
product of the ray class groups modulo $\cM$ of the Wedderburn
components $K(t)$ (see \eqref{E:wiso}) of $K\Lambda$. \qed
\end{definition}

The following result shows that each coset $\lambda \cP_{\theta}$ of
$\cP_{\theta}$ in $J(K \Lambda)$ is a disjoint union of cosets of
$U_{\cM}'(\Lambda) \cdot K \Lambda$ in $J(K(\Lambda))$ for any suitably
chosen ideal $\cM$ of $O_K$.

\begin{proposition} \label{P:finiteness}
Let $\cM$ be an integral ideal of $O_K$ that is divisible by both
$|G|$ and $\exp(G)^2$ (where $\exp(G)$ denotes the exponent of
$G$). Then there is a natural quotient homomorphism
\begin{equation*}
f_{\cM}: \Cl_{\cM}'(\Lambda) \to \frac{J(K\Lambda)}{\cP_{\theta}}.
\end{equation*}
In particular, the group $J(K\Lambda)/\cP_{\theta}$ is finite.
\end{proposition}

\begin{proof} Set
\begin{equation*}
\cP_{\cM}:= (K \Lambda)^{\times} \cdot U_{\cM}'(\Lambda) \subseteq J(K
\Lambda)
\end{equation*}
McCulloh has shown (see \cite[Proposition 6.9]{Mc}) that if $\cM$ is
divisible by both $|G|$ and $\exp(G)^2$, then 
\begin{equation*}
\Theta^t(\cP_{\cM}) \subseteq \cH(\bA(O_KG)),
\end{equation*}
whence it follows from the definition of $\theta$ that
$\theta(\cP_{\cM}) =0$. This implies that
\begin{equation*}
\cP_{\cM} \subseteq \cP_{\theta} \subseteq J(K\Lambda),
\end{equation*}
and so there is a natural quotient homomorphism $f_{\cM}$, as
asserted. Since $\Cl_{\cM}'(\Lambda)$ is finite, it follows that the
same is true of $J(K\Lambda)/\cP_{\theta}$.
\end{proof}

Let $I(\Lambda)$ denote the group of fractional ideals of
$\Lambda$. Via the Wedderburn decomposition \eqref{E:wiso} of
$\Lambda$, each ideal $\fA$ in $I(\Lambda)$ may be written $\fA=
(\fA_t)_{t \in T}$, where each $\fA_t$ is a fractional ideal of
$O_{K(t)}$.

For any idele $\lambda \in J(K\Lambda)$, we write $\co(\lambda) \in
I(\Lambda)$ for the ideal obtained by taking the idele content of
$\lambda$. The following proposition describes exactly which ideals in
$I(\Lambda)$ arise via taking the idele content of elements in $\bF
\subseteq J(K\Lambda)$.

\begin{proposition} \label{P:fideals}
Let $\fF$ be the subset of $I(\Lambda)$ defined by
\begin{equation*}
\fF =\{ \co(f) \mid f \in \bF \}.
\end{equation*}
Then $\fF$ consists precisely of those ideals $\ff = (\ff_t)_{t \in T}$
such that

$\bullet$ $\ff_1 = O_K$;

$\bullet$ $N_{K(\Lambda)/K}(\ff) := \prod_{t \in T} N_{K(t)/K}(\ff_t)$
is a squarefree $O_K$-ideal;

$\bullet$ $\ff_t$ is coprime to the order $|t|$ of $t$.

In particular, if we view $\bF_v$ as being a subset of $\bF$ via the
obvious embedding $(K_v \Lambda)^{\times} \subseteq J(K\Lambda)$, 
then
\begin{equation*}
\fF_v:= \{ \co(f_v) \mid f_v \in \bF_v \}
\end{equation*}
consists precisely of the invertible prime ideals of $\Lambda$ arising
via \eqref{E:wiso} from the invertible prime ideals of relative degree
one over $v$ in those Wedderburn components $K(t)$ of $\Lambda$ for
which $t \neq 1$ and $v(|t|) = 0$.
\end{proposition}

\begin{proof} 
See \cite[pages 288-289]{Mc}.
\end{proof} 

\begin{example} \label{E:keyexample}
Suppose that $h \in H^{1}_{tr}(K,G)$. Recall (see Remark
\ref{R:newleon}) that there exist unique $c \in \cR(O_KG)$ and $f \in
\bF$ such that $\rho(c) = \psi(h)^{-1} \theta(f)$. Let
\begin{equation*}
\co(f) = \ff = (\ff_t)_{t \in T}.
\end{equation*}
Then each ideal $\ff_t$ of $O_{K(t)}$ may be written as a product
\begin{equation*}
\ff_t = \cP_{t,1} \cdots \cP_{t,i_t}
\end{equation*}
of primes of relative degree one in $K(t)/K$. Each finite place
$v$ of $K$ that ramifies in $K_h/K$ lies beneath exactly one ideal
$\cP_{t,j}$, and in this case the ramification index of $v$ in $K_h/K$
is equal to the order $|t|$of $t$ (see \cite[Proposition 5.4]{Mc}). It therefore
follows from the standard formula for tame discriminants that
\begin{equation*}
\disc(K_h/K) = \prod_{t \in T} N_{K(t)/K} (\ff_t)^{(|t|-1)|G|/|t|}.
\end{equation*}
Hence the absolute norm $\cD(K_h/K)$ of $\disc(K_h/K)$ is given by
\begin{equation*}
\cD(K_h/K) = \left[O_K: \prod_{t \in T} N_{K(t)/K}
  (\ff_t)^{(|t|-1)|G|/|t|}\right].
\end{equation*}

Let $d(\ff) = (d(\ff_t))_{t \in T}$ denote the ideal in $I(\Lambda)$
defined by $d(\ff)_{1} = O_K$ and 
\begin{equation*}
d(\ff)_t = \ff_{t}^{(|t|-1)|G|/|t|}
\end{equation*}
for $t \neq 1$. Then since
\begin{equation*}
[O_{K(t)}: \ff_t] = [O_K: N_{K(t)/K}(\ff_t)],
\end{equation*}
for each $t \neq 1$, it follows that we have
\begin{equation*}
\cD(K_h/K) = [\Lambda: d(\ff)].
\end{equation*}
\qed
\end{example}

Example \ref{E:keyexample} motivates the following definitions.

\begin{definition} \label{D:weight}
We say that a function
\begin{equation*}
\cW: T \to \bZ_{\geq 0}
\end{equation*}
is a \textit{weight function on $T$} (or just a \textit{weight} for
short) if $\cW(1)=0$ and $\cW(t) \neq 0$ for all $t \neq 1$.

For any weight $\cW$, we set
\begin{equation*}
\alpha_{\cW} := \min\{\cW(t): t \neq 1 \}.
\end{equation*}
\qed
\end{definition}

\begin{definition} \label{D:idweight}
Suppose that $\cW$ is a weight and $\fA =
  (\fA_t)_{t \in T}$ is an ideal in $I(\Lambda)$. We write
  $d_{\cW}(\fA) = (d_{\cW}(\fA)_t)_{t \in T}$ for the ideal in
  $I(\Lambda)$ defined by $d_{\cW}(\fA)_t = \fA_{t}^{\cW(t)}$. 
\qed
\end{definition}

\begin{definition} \label{D:discweight}
Suppose that $h \in H^{1}_{tr}(K,G)$ with $\rho(c) = \psi(h)^{-1}
\theta(f)$. For any weight function $\cW$ on $T$, we set
\begin{equation} \label{E:discweight}
D_{\cW}(K_h/K):= [\Lambda: d_{\cW}(\co(f))].
\end{equation}
\end{definition}

\begin{example} \label{E:exweight} Let $K_h/K$ be any tamely ramified
  Galois $G$-extension of $K$.
\smallskip

(1) Define a weight function $\cW_{\disc}$ on $T$ by $\cW_{\disc}(t)=
    (|t|-1)|G|/|t|$ for $t \neq 1$. Then we see from Example
    \ref{E:keyexample} that $D_{\cW_{\disc}}(K_h/K)$ is equal to the
    absolute norm of the relative discriminant of $K_h/K$.
\smallskip

(2) Define a weight function $\cW_{\ram}$ on $T$ by $\cW_{\ram}(t) =
    1$ for $t \neq 1$. Then $D_{\cW_{\ram}}(K_h/K)$ is equal to the
    absolute norm of the product of the primes of $K$ that are
    ramified in $K_h/K$.
\qed
\end{example}

We now fix once and for all an integral ideal $\cM$ of $O_K$ that is
divisible by both $|G|$ and $\exp(G)^2$, and we also fix a weight
function $\cW$ on $T$.

\begin{definition}
For each $c \in \cR(O_KG)$ and each real number $X>0$, we write
$N_{\cW}(c,X;\cM)$ for the number of tame Galois $G$-extensions
$K_h/K$ for which $(O_h) = c$, $D_{\cW}(K_h/K)$ is coprime to $\cM$,
and $D_{\cW}(K_h/K) \leq X$. 

We define $M_{\cW}(X;\cM)$ to be the number of tame Galois
$G$-extensions $K_h/K$ for which $D_{\cW}(K_h/K) \leq X$ and
$D_{\cW}(K_h/K)$ is coprime to $\cM$.
\end{definition}

\begin{question} \label{Q:general}
What can be said about the behaviour of $N_{\cW}(c,X;\cM)$ as $X \to
\infty$? For example, is
\begin{equation*}
Z_{\cW}(c;\cM):=\lim_{X \to \infty} \frac{N_{\cW}(c,X;\cM)}{M_{\cW}(X;\cM)}
\end{equation*}
independent of $c$? \qed
\end{question}

For each coset $\fc$ of $\cP_{\cM}$ in $J(K \Lambda)$, set
\begin{equation*}
\kappa_{\cW}(\fc,X;\cM) = |\left\{ f \in \bF \cap \fc \mid \text{
  $(\co(f), \cM)=1$ and $[\Lambda:
  d_{\cW}(\co(f))] \leq X$} \right\}|
\end{equation*}
Then it follows from Remark \ref{R:newleon}(1) and Proposition
\ref{P:pprop} that there is a unique coset $\lambda_c \cP_{\theta}$ of
$\cP_{\theta}$ in $J(K \Lambda)$ such that
\begin{align} \label{E:countfunc}
N_{\cW}(c,X;\cM) &= |\Ker(\psi)| \cdot | \{ f \in F \cap \lambda_c
\cP_{\theta} \mid \text{$(\co(f),\cM)=1$ and $[\Lambda:
    d_{\cW}(\co(f))] \leq X$} \} | \notag \\ 
&= |\Ker(\psi)| \cdot \sum_{\fc \in f_{\cM}^{-1}(c)} \kappa_{\cW}(\fc,
X;\cM).
\end{align} 
We therefore see that the behaviour of $N_{\cW}(c,X;\cM)$ as $X \to
  \infty$ is governed by that of the $\kappa_{\cW}(\fc, X;\cM)$. For
  example, if $\kappa_{\cW}(c,X;\cM)$ is asymptotically independent of
  $\fc$ (see Definition \ref{D:ind} below), then it follows that
  asymptotically, $N_{\cW}(c,X;\cM)$ is independent of the realisable
  class $c \in \cR(O_KG)$.

%%%%%%%%%%%%%%%%%%%%%%%%%%%%%%%%%%%%%%%%%%%%%%%%%%%%%%%%%%%%%%%%%%%%%%%%%%%%%%

\section{Euler Products} \label{S:euler}

Recall (see Proposition \ref{P:fideals}) that $\fF$ denotes the subset
of $I(\Lambda)$ defined by 
\begin{equation*}
\fF =\{ \co(f) \mid f \in \bF \}.
\end{equation*}

\begin{definition} \label{D:dfunc}
We define functions $D(s)$ and $D_{\cM}(s)$ of a complex variable $s$
by
\begin{equation}
D(s):= \sum_{\fa \in \fF} [\Lambda:d_{\cW}(\fa)]^{-s};\qquad 
D_{\cM}(s):= \sum_{\substack{ \fa \in \fF \\ (\fa,
    \cM)=1}}[\Lambda:d_{\cW}(\fa)]^{-s}.
\end{equation}

For any $\fc \in \Cl_{\cM}'(\Lambda)$, we set
\begin{equation}
 D_{\fc}(s):= \sum_{\fa \in \fF \cap \fc} [\Lambda:d_{\cW}(\fa)]^{-s};\qquad
D_{\fc,\cM}(s):= \sum_{\substack{ \fa \in \fF \cap \fc \\ (\fa,
    \cM)=1}}[\Lambda:d_{\cW}(\fa)]^{-s}.
\end{equation}

Each of the functions above also depends upon the choice of $\cW$; we
omit this dependence from our notation.
\qed
\end{definition}

Let $\chi$ be any character of $\Cl'_{\cM}(\Lambda)$, and set $T':= T
\backslash \{1\}$. Then via the Wedderburn decomposition
\eqref{E:wiso} of $\Lambda$, we may write $\chi =(\chi_t)_{t \in
T'}$, where each $\chi_t$ is a character of the ray class group modulo
$\cM$ of $K(t)$. We may view $\chi$ as being a map on the set of all
integral ideals $\fa = (\fa_t)_{t \in T}$ in the standard manner by
setting $\chi(\fa) = 0$ if $\fa_1 \neq O_K$ or if $\fa$ is not coprime
to $\cM$.

\begin{definition} \label{D:chidfunc}
For each character $\chi$ of $\Cl_{\cM}'(\Lambda)$, we define
\begin{equation}
D(s,\chi) = \sum_{\fa \in \fF} \chi(\fa) [\Lambda:d_{\cW}(\fa)]^{-s}.
\end{equation}
\qed
\end{definition}

With the above definitions, we have
\begin{equation} \label{E:dfourier}
D_{\fc, \cM}(s) = \frac{1}{|\Cl_{\cM}'(\Lambda)|} \sum_{\chi}
\overline{\chi}(\fc)D(s,\chi),
\end{equation}
where the sum is over all characters $\chi$ of $\Cl_{\cM}'(\Lambda)$.

\begin{definition} \label{D:index} (cf. \cite[Chapter I]{CF})
Let $\fa = (\fa_t)_{t \in T}$ be any ideal in $I(\Lambda)$. We define
the \textit{module index} $[\Lambda: \fa]_{O_K}$ to be the $O_K$-ideal
given by
\begin{equation}
[\Lambda: \fa]_{O_K}:= \prod_{t \in T} N_{K(t)/K}(\fa_t).
\end{equation}
\qed
\end{definition}

\begin{lemma} \label{L:numult}
For each integral $O_K$-ideal $\fb$, set
\begin{equation*}
\nu(\fb):= |\{ \fa \in \fF \mid [\Lambda:d_{\cW}(\fa)]_{O_K} = \fb
\}|.
\end{equation*}
Then $\nu$ is multiplicative, i.e. if $\fb_1, \fb_2$ are coprime
$O_K$-ideals, we have
\begin{equation*}
\nu(\fb_1 \fb_2) = \nu(\fb_1) \nu(\fb_2).
\end{equation*}
\end{lemma}

\begin{proof}
It follows from Proposition \ref{P:fideals} that if $\fa_1, \fa_2$ are
in $\fF$, and $[\Lambda:d_{\cW}(\fa_1)]_{O_K}$ and
$[\Lambda:d_{\cW}(\fa_2)]_{O_K}$ are coprime, then $\fa_1 \fa_2$ lies
in $\fF$ also. Hence, for any choice of ideals $\fa_1, \fa_2 \in \fF$
with $[\Lambda:d_{\cW}(\fa_i)]_{O_K} = \fb_i$ $(i=1,2)$, we have
\begin{align*}
[\Lambda: d_{\cW}(\fa_1 \fa_2)]_{O_K} &=[\Lambda:d_{\cW}(\fa_1)]_{O_K}
\cdot [\Lambda:d_{\cW}(\fa_2)]_{O_K} \\ 
&= \fb_1 \cdot \fb_2,
\end{align*}
and so we deduce that $\nu(\fb_1 \fb_2) \geq \nu(\fb_1) \nu(\fb_2)$.

To show the reverse inequality, set $\fb = \fb_1 \fb_2$, and let $\fa
\in\fF$ be any ideal such that $[\Lambda: d_{\cW}(\fa)]_{O_K} =
\fb$. For each $i =1,2$, let $\fa_i$ be the product of all primes
$\fP$ of $\Lambda$ with $\fP$ a prime factor of $\fa$ and $[\Lambda:
\fP]_{O_K}$ a prime factor of $\fb_i$. Then we have
\begin{equation} \label{E:splitting}
\fa = \fa_1 \fa_2,\quad \fa_i \in \fF, \quad \text{and}\quad [\Lambda:
  \fa_i]_{O_K} = \fb_i,\quad (i = 1,2).
\end{equation}
Furthermore, it follows via uniqueness of factorisation in $\Lambda$
and $O_K$ that $\fa_1$ and $\fa_2$ are the unique ideals satisfying
\eqref{E:splitting}. This implies that $\nu(\fb_1 \fb_2) \leq
\nu(\fb_1) \nu(\fb_2)$, and so we finally deduce that $\nu(\fb_1
\fb_2)= \nu(\fb_1) \nu(\fb_2)$  as asserted.
\end{proof}

\begin{proposition} \label{P:euler}
The functions $D(s)$ and $D(s, \chi)$ admit Euler product expansions
over the rational primes:
\begin{equation*}
D(s) = \prod_p D_p(s), \qquad D(s, \chi)= \prod_p D_p(s, \chi).
\end{equation*}
\end{proposition}

\begin{proof}
Suppose that $\fa \in \fF$, with $[\Lambda: d_{\cW}(\fa)]_{O_K} =
\fb$. Then it follows from Proposition \ref{P:fideals} that
\begin{equation*}
[\Lambda: d_{\cW}(\fa)] = [O_K: \fb].
\end{equation*}
This in turn implies that
\begin{align*}
D(s) &= \sum_{\fa \in \fF} [\Lambda: d_{\cW}(\fa)] \\
&= \sum_{\substack{\fb \in I(O_K)\\ \fb \subseteq O_K}} \nu(\fb)[O_K:\fb]^{-s}.
\end{align*}
Since $\nu$ is multiplicative, we have
\begin{equation*}
D(s) = \prod_{\substack{ \fp \in I(O_K)\\ \text{$\fp$ prime} } }
D_{\fp}(s),
\end{equation*}
where
\begin{equation*}
D_{\fp}(s) = 1 + \sum_{m=1}^{\infty} \nu(\fp^m)[O_K:\fp]^{-ms}.
\end{equation*}

Next, we observe that since $\fa \in \fF$ implies that $\fa$ is
squarefree (see Proposition \ref{P:fideals}), it follows that we can
find a positive integer $N$, say, independent of $\fp$, such that
$\nu(\fp^{m}) = 0$ for all $m >N$. (In fact $N = |G| \cdot \max\{
\cW(t) \mid t \in T\}$ will do.) We may therefore write
\begin{equation*}
D_{\fp}(s) = 1 + \sum_{m=1}^{N} \nu(\fp^m)[O_K:\fp]^{-ms},
\end{equation*}
and we define $D_p(s)$ by
\begin{equation*}
D_p(s) = \prod_{\fp \mid p} D_{\fp}(s).
\end{equation*}
Thus we see that
\begin{equation*}
D(s) = \prod_p D_p(s),
\end{equation*}
as claimed.

We now show that $D(s, \chi)$ also admits an Euler product
expansion. For each rational prime $p$, set
\begin{equation*}
\fF(p):= \left\{ \fa \in \fF \mid \text{$[\Lambda: \fa]$ is a non-negative
  power of $p$} \right\}.
\end{equation*}
Observe that $\fa \in \fF(p)$ if and only if all prime factors of
$\fa$ in $\Lambda$ lie above $p$, and we have that
\begin{equation*}
D_p(s) = 1 + \sum_{\fa \in \fF(p)} [\Lambda:d_{\cW}(\fa)]^{-s}.
\end{equation*}
A very similar argument to that given above now shows that
\begin{equation*}
D(s, \chi)= \prod_p D_p(s,\chi),
\end{equation*}
where
\begin{equation} \label{E:Deuler}
D_p(s,\chi) = 1 + \sum_{\fa \in \fF(p)} \chi(\fa)
[\Lambda:d_{\cW}(\fa)]^{-s}.
\end{equation}
This establishes the desired result.
\end{proof}

%%%%%%%%%%%%%%%%%%%%%%%%%%%%%%%%%%%%%%%%%%%%%%%%%%%%%%%%%%%%%%%%%%%%%%%%%%
\section{The asymptotic behaviour of $\kappa_{\cW}(\fc,
  X;\cM)$} \label{S:asymptotics} 

In this section we shall obtain an expression for
\begin{equation*}
\kappa_{\cW}(\fc,X;\cM) := |\left\{ f \in \bF \cap \fc \mid \text{
  $(\co(f), \cM)=1$ and $[\Lambda:
  d_{\cW}(\co(f))] \leq X$} \right\}|
\end{equation*}
for each $\fc \in J(K\Lambda)/\cP_{\cM}$ when $X$ is large. We shall do this by
appealing to the following version of the D\'elange-Ikehara Tauberian
theorem.

\begin{theorem}  \label{T:delange}
Suppose that $f(s) = \sum_{n=1}^{\infty} a_n n^{-s}$ is a Dirichlet
series with non-negative coefficients, and that it is convergent for
$\Re(s)>a>0$. Assume that in its domain of convergence,
\begin{equation*}
f(s) = g(s)(s-a)^{-w} + h(s)
\end{equation*}
holds, where $g(s), h(s)$ are holomorphic functions in the closed
half-plane $\Re(s) \geq a$, $g(a) \neq 0$, and $w>0$. Then, as $X \to
\infty$, we have
\begin{equation*}
\sum_{n \leq X} a_n \sim \frac{g(a)}{a \cdot \Gamma(w)}
 \cdot X^a \cdot (\log X)^{w-1}.
\end{equation*}
\end{theorem}

\begin{proof} 
See \cite[p. 121]{Nark}.
\end{proof}

We see from \eqref{E:dfourier} that each function $D_{\fc, \cM}(s)$ is
convergent in some right-hand half-plane, because $D(s,\chi)$ has an
Euler product expansion for all characters $\chi$ of
$\Cl_{\cM}'(\Lambda)$. It also follows from the definitions that each
$D_{\fc, \cM}(s)$ is a Dirichlet series with non-negative
coefficients. If we write
\begin{equation*}
D_{\fc,\cM}(s) = \sum_{n=0}^{\infty} a_n n^{-s}.
\end{equation*}
then we have
\begin{equation*}
\kappa_{\cW}(\fc,X;\cM)= \sum_{n \leq X} a_n.
\end{equation*}

For each $\fc \in J(K\Lambda)/\cP_{\cM}$, let $\beta(\fc; \cM)$ denote
right-most pole of $D_{\fc,\cM}(s)$ in the complex plane. It follows
from a theorem of Landau that $\beta(\fc; \cM)$ is real (see
\cite[Theorem 3.5]{Nark}). Let $\delta(\fc; \cM)$ denote the
order of this pole. Write
\begin{equation*}
\tau(\fc;\cM):= \lim_{s \to \beta(\fc;\cM)} (s-
\beta(\fc;\cM))^{\delta(\fc;\cM)} D_{\fc,\cM}(s).
\end{equation*}

\begin{proposition} \label{P:kappa}
As $X \to \infty$, we have
\begin{equation*}
\kappa_{\cW}(\fc, X;\cM) \sim  \frac{\tau(\fc;\cM)}{\beta(\fc;\cM)
  \cdot \Gamma(\delta(\fc;\cM))}  \cdot X^{\beta(\fc;\cM)}
\cdot
(\log X)^{\delta(\fc;\cM)-1}.
\end{equation*}
\end{proposition}

\begin{proof} 
This follows directly from Theorem \ref{T:delange}.
\end{proof}

\begin{definition} \label{D:ind}
If
\begin{equation} \label{E:ind}
\kappa_{\cW}(\fc_1,X;\cM) \sim \kappa_{\cW}(\fc_2,X;\cM)
\end{equation}
as $X \to \infty$ for all $\fc_1, \fc_2 \in \Cl'_{\cM}(\Lambda)$, then
we shall say that \textit{$\kappa_{\cW}(\fc,X;\cM)$ is asymptotically
  independent of $\fc$}.

It is not hard to see that \eqref{E:ind} holds for all $\fc_1, \fc_2
\in \Cl_{\cM}'(\Lambda)$ if and only if the numbers $\tau(\fc;\cM),
\beta(\fc;\cM)$, and $\delta(\fc;\cM)$ do not vary with $\fc$. \qed
\end{definition}

We shall see in Section \ref{S:poles} that, in general,
$\kappa_{\cW}(\fc,X;\cM)$ is not asymptotically independent of $\fc$.

%%%%%%%%%%%%%%%%%%%%%%%%%%%%%%%%%%%%%%%%%%%%%%%%%%%%%%%%%%%%%%%%%%%%%%%%%%

\section{Dirichlet $L$-series} \label{S:dirichlet}

We now turn our attention to certain Dirchlet $L$-series associated to
$\Lambda$. These will be used in the next section to study the
behaviour of the functions $D(s)$ and $D(s,\chi)$.

\begin{definition} \label{D:dirichlet}
Suppose that $\chi = (\chi_t)_{t \in T'}$ is a character of
$\Cl_{\cM}'(\Lambda)$. We define
\begin{equation*}
L_{\Lambda}(s, \chi):= \sum_{\substack{ \fa \in I(\Lambda) \\
\fa \subseteq \Lambda }}
\chi(\fa)[\Lambda: d_{\cW}(\fa)]^{-s}.
\end{equation*}
\qed
\end{definition}

\begin{remark} \label{R:dirichlet}
(1) For each character $\chi = (\chi_t)_{t \in T'}$ of
  $\Cl_{\cM}'(\Lambda)$, the function $L_{\Lambda}(s, \chi)$ is a
  product of $L$-functions of number fields. If we set
\begin{equation*}
L_t(s, \chi_t) = \sum_{\substack{ \fb \in I(O_{K(t)}) \\ \fb \subseteq
    O_{K(t)}}} \chi_t(\fb) \fb^{-\cW(t)s},
\end{equation*}
then corresponding to the Wedderburn decomposition \eqref{E:wiso} of
$K \Lambda$, we have 
\begin{equation} \label{E:liso}
L_{\Lambda}(s, \chi) = \prod_{t \in T'} L_t(s, \chi_t).
\end{equation}

It follows from standard properties of Dirichlet $L$-series that
$L_t(\frac{1}{\cW(t)}, \chi_t) \neq 0$ if $\chi_t \neq 1$ and that
$L_t(s,\1_t)$ has a simple pole at $s = 1/\cW(t)$.
\smallskip

(2) The function $L_{\Lambda}(s,\chi)$ has an Euler product given by 
\begin{equation*}
L_{\Lambda}(s, \chi) = \prod_{p} L_{\Lambda, p}(s, \chi),
\end{equation*}
where 
\begin{equation*}
L_{\Lambda, p}(s, \chi) = 1 + \sum_{\fa}
\chi(\fa)[\Lambda:d_{\cW}(\fa)]^{-s};
\end{equation*}
here the sum is over all integral ideals $\fa$ of $\Lambda$ lying above the
rational prime $p$.

Let $P_1, \ldots, P_{n(p)}$ be the invertible primes of $\Lambda$
which lie above the rational prime $p$. (Note that the integer $n(p)$
is bounded above independently of $p$.) Then we also have
\begin{equation*}
L_{\Lambda, p}(s,\chi)= \prod_{i=1}^{n(p)}(1-\chi(P_i)[\Lambda:
  d_{\cW}(P_i)]^{-s})^{-1}.
\end{equation*}
\qed
\end{remark}

In Section \ref{S:poles} we shall compare the functions
$L_{\Lambda}(s,\chi)$ and $D(s,\chi)$ by examining corresponding terms
in their Euler product expansions. In order to do this, we shall need
the following two technical lemmas from \cite{Fo}.

\begin{lemma} \label{L:tech1}
\cite[Lemma 1.1]{Fo}.
Expand 
\begin{equation*}
F(z_1,...,z_n):= \prod_{i=1}^{n} (1-z_i)^{-1}
\end{equation*}
as an infinite series of monomials in $z_1,..,z_n$. Suppose that $0 <
r \leq r_0 < 1$, and that there is a positive integer $m \leq n$ such
that $|z_i| \leq r$ and $i \leq m$ and $|z_i| < r^2$ for $i > m$.

Then, if $f(z_1,...,z_n)$ is any subseries of the series for
$F(z_1,..,z_n)$ containing the terms $1 + \sum_{i=1}^{m} z_i$, we have
\begin{equation*}
|F(z_1,...,z_n) - f(z_1,...,z_n)| \leq \left[
  \frac{n(n+1)}{2(1-r_0)^{n+2}} + n \right] r^2.
\end{equation*}
\end{lemma}

\begin{proof} Since the series for $F-f$ has only positive
  coefficients, it follows that an upper bound for $|F-f|$ may be
  obtained by setting $z_i = r$ for $i \leq m$, and $z_i = r^2$ for
  $i>m$, and by replacing $f(z_1,...,z_n)$ with $1+ \sum_{i=1}^{m}
  z_i$.

For the terms of degree one in $F-f$, we have
\begin{equation*}
\left| \sum_{i=m+1}^{n} z_i \right| \leq nr^2.
\end{equation*}
Also, as each term of degree $k$ with $k \geq 2$ has absolute value at
most $r^k$, and there are ${n+1-k} \choose k$ such terms,
it follows that the sum of all such terms (for all $k \geq
2$) has absolute value at most $(1-r)^{-n} - (1+nr)$. 
By comparing the terms in the binomial expansions of
\begin{equation*}
h_1(x) = (1-x)^{-n} - (1+nx),\qquad h_2(x) =
\frac{n(n+1)}{2(1-x)^{n+2}} x^2
\end{equation*}
we see that the inequality
\begin{equation*}
0 < (1-r)^{-n} - (1+nr) \leq \frac{n(n+1)}{2(1-r)^{n+2}} \cdot r^2
\end{equation*}
holds. Therefore, since $r \leq r_0 < 1$, we obtain
\begin{equation*}
|F(z_1,...z_n) - f(z_1,...z_n)| \leq  \frac{n(n+1)}{2(1-r)^{n+2}}
\cdot r^2 + nr^2.
\end{equation*}
This completes the proof.
\end{proof}

\begin{lemma} \label{L:tech2}
\cite[Lemma 1.2]{Fo}.
Let $\phi(s)$ and $\phi^*(s)$ be Dirichlet series with Euler products
\begin{equation*}
\phi(s) = \prod_p \phi_p(s),\qquad \phi^*(s) = \prod_p \phi_p^*(s)
\end{equation*}
over the rational primes. Suppose that $\phi(s)$ and $\phi^*(s)$ are
absolutely convergent for $\Re(s) >1$.

Suppose further that:

(i) For every $p$, $\phi_p(s)$ and $\phi^*_p(s)$ are analytic for
$\Re(s)>0$;

(ii) Given a real number $\sigma_0$ with $0 < \sigma_0 < 1$, there
exists $B(\sigma_0) = B >0$ such that
\begin{equation*}
\left| \frac{\phi_p^*(s) - \phi_p(s)}{\phi_p^*(s)} \right| < B
\cdot p^{-2\sigma_0}
\end{equation*}
for every $p$ and $\sigma = \Re(s) \geq \sigma_0$.

Then $\phi(s) = \phi^*(s) \psi(s)$, where $\psi(s)$ is analytic for
$\Re(s)> 1/2$. If $z \in \bC$ satisfies $\Re(z)>1/2$, and if
 $\phi_p(z) \neq 0$ for all $p$, then $\psi(z) \neq 0$.
\end{lemma}

\begin{proof}
We first observe that (i) implies that $\phi_p(s)/\phi^*_p(s)$ is
meromorphic for $\Re(s)>0$, and so it follows from (ii) that in fact
$\phi_p(s)/\phi^*_p(s)$ is analytic for $\Re(s)>0$. For $\Re(s)>1$,
define
\begin{equation*}
\psi(s) = \prod_p \frac{\phi_p(s)}{\phi_p^*(s)} = \prod_p \left[ 1 -
  \frac{\phi_p^*(s) - \phi_p(s)}{\phi_p^*(s)} \right].
\end{equation*}
We see from (ii) that this product converges whenever $\sum_p
p^{-2\sigma}$ converges, i.e. for $\Re(s) = \sigma > 1/2$. This
implies that $\psi(s)$ is analytic for $\Re(s)> 1/2$.

It is easy to verify that we have $\phi(s) = \phi^*(s) \psi(s)$ as a
formal identity. If $\phi_p(z) \neq 0$ for all $p$, then none of the
factors of $\psi(z)$ are zero. Since the product defining $\psi(z)$ is
absolutely convergent, it follows that $\psi(z) \neq 0$, as claimed.
\end{proof}

%%%%%%%%%%%%%%%%%%%%%%%%%%%%%%%%%%%%%%%%%%%%%%%%%%%%%%%%%%%%%%%%%%%%%%%%%%%%%%%
 
\section{The poles of $D(s,\chi)$ and $D_{\fc,\cM}(s)$} \label{S:poles}

In this section, using techniques described in \cite{Fo}, we shall
examine the poles of $D(s, \chi)$ and $D_{\fc,\cM}(s)$. We shall do
this by comparing the Euler product expansion of $D(s,\chi)$ with that
of $L_{\Lambda}(s, \chi)$ and applying Lemmas \ref{L:tech1} and
\ref{L:tech2}.

\begin{proposition} \label{P:comp1} For each rational prime $p$ with
  $p \nmid \cM$, we have
\begin{equation*}
| L_{\Lambda,p}(s,\chi) - D_p(s,\chi)| \leq \left[
  \frac{n(n+1)}{(1-2^{-\sigma_{0}} )^{n+2}} + n \right] p^{-2
  \alpha_{\cW} \Re(s)},
\end{equation*}
for any real number $\sigma_0$ satisfying $0<\sigma_o < \alpha_{\cW}
\Re(s)$. Here $n=n(p)$, the number of primes of $\Lambda$ lying above
$p$.
\end{proposition}

\begin{proof}
We first observe that the series defining $D_p(s,\chi)$ is a subseries
of the series defining $L_{\Lambda,p}(s,\chi)$. Also, the series
defining $D_p(s,\chi)$ contains the terms
\begin{equation*}
1 + \sum_{i=1}^{m} \chi(P_i)[\Lambda: d_{\cW}(P_i)]^{-s},
\end{equation*}
where the $P_i$ are arranged so that $P_1,...,P_m$ satisfy
$[\Lambda:P_i] = p$, and $P_{m+1},...,P_n$ satisfy $[\Lambda:P_i] \geq
p^2$. 

In Lemma \ref{L:tech1}, we take
\begin{equation*}
z_i:= \chi(P_i)[\Lambda:d_{\cW}(P_i)]^{-s},\qquad F(z_1,...,z_n):=
L_{\Lambda,p}(s,\chi), \qquad f(z_1,...,z_n):= D_p(s,\chi).
\end{equation*}
We observe that, for $1 \leq i \leq n$, we have
\begin{equation*}
[\Lambda: d_{\cW}(P_i)] \geq p^{\alpha_{\cW}},
\end{equation*}
and so 
\begin{align*}
| [\Lambda:d_{\cW}(P_i)]^{-s} | &\geq |p^{-\alpha_{\cW} s}| \\
&= p^{-\alpha_{\cW} \Re(s)}.
\end{align*}

Hence, if we set $r= p^{-\alpha_{\cW} \Re(s)}$ and $r_0=2^{-\sigma_{0}}$ with
$0< \sigma_0 \leq \alpha_{\cW} \Re(s)$, then we have $0<r \leq r_0<1$,
$|z_i| \leq r$ for $1 \leq i \leq m$, and $|z_i| \leq r^2$ for $m+1
\leq i \leq n$. So, the conditions of Lemma \ref{L:tech1} are
satisfied, and we have
\begin{equation*}
| L_{\Lambda,p}(s,\chi) - D_p(s,\chi)| \leq \left[
  \frac{n(n+1)}{(1-2^{-\sigma_{0}} )^{n+2}} + n \right] p^{-2
  \alpha_{\cW} \Re(s)},
\end{equation*}
as claimed.
\end{proof}

\begin{proposition} \label{P:comp2} 
For each character $\chi = (\chi_t)_{t \in T'}$ of
$\Cl'_{\cM}(\Lambda)$, we may write  
\begin{equation*}
D(s,\chi) = L_{\Lambda}(s, \chi) \cdot \psi(s,\chi),
\end{equation*}
where $\psi(s,\chi)$ is analytic for $\Re(s) > 1/(2\alpha_{\cW})$.

If $z \in \bC$ satisfies $\Re(z)> 1/(2\alpha_{\cW})$, and $D_p(z,\chi) \neq 0$
for all $p$, then $\psi(z,\chi) \neq 0$. 
\end{proposition}

\begin{proof}
To prove the desired result, we are going to apply Lemma \ref{L:tech2}
with
\begin{equation*}
\phi(s) = D(s,\chi),\qquad \phi^*(s) = L_{\Lambda}(s,\chi).
\end{equation*}

We first note that for each prime $p$ with $p \nmid \cM$, the Euler
factor $L_{\Lambda,p}(s,\chi)$ is analytic for $\Re(s)>0$. This
implies that $D_p(s,\chi)$ is also analytic for $\Re(s)>0$, because
the series defining $D_p(s,\chi)$ is a subseries of the series
defining $L_{\Lambda,p}(s,\chi)$.

Set $N:= \dim_{\bQ}(K \Lambda)$. We have
\begin{align*} 
|L_{\Lambda,p}(s,\chi)|^{-1} &= \prod_{i=1}^{n(p)}
|(1-\chi(P_i)[\Lambda:d_{\cW}(P_i)]^{-s}) | \\
&\leq (1+p^{-\alpha_{\cW} \Re(s)})^N.
\end{align*}
In particular, this implies that
\begin{equation} \label{E:Lbound}
|L_{\Lambda,p}(s,\chi)|^{-1} \leq (1 + 2^{-\sigma_{0}})^N
\end{equation}
for all $p \nmid \cM$ and for all $s \in \bC$ with $\alpha_{\cW} \Re(s) \geq
\sigma_0$. Applying Proposition \ref{P:comp1} gives
\begin{equation*}
|L_{\Lambda,p}(s,\chi)| \cdot \left| \frac{L_{\Lambda,p}(s,\chi) -
D_p(s,\chi)}{L_{\Lambda,p}(s,\chi)} \right| 
\leq \left[
  \frac{n(n+1)}{(1-2^{-\sigma_{0}} )^{n+2}} + n \right] p^{-2
  \alpha_{\cW} \Re(s)}.
\end{equation*}
We therefore see from \eqref{E:Lbound} that
\begin{align*}
\left| \frac{L_{\Lambda,p}(s,\chi) -
D_p(s,\chi)}{L_{\Lambda,p}(s,\chi)} \right| 
&\leq \left[
  \frac{n(n+1)}{(1-2^{-\sigma_{0}} )^{n+2}} + n \right] p^{-2
  \alpha_{\cW} \Re(s)} \cdot (1+2^{-\sigma_{0}})^N \\
&= B(\sigma_0) \cdot p^{-2 \alpha_{\cW} \Re(s)},
\end{align*}
say. Hence condition (ii) of Lemma \ref{L:tech2} is satisfied, but with
$\sigma = \alpha_{\cW} \Re(s)$, rather than $\sigma = \Re(s)$.

Lemma \ref{L:tech2} therefore implies that we may write
\begin{equation*}
D(s,\chi) = L_{\Lambda}(s,\chi) \cdot \psi(s,\chi),
\end{equation*}
where $\psi(s,\chi)$ is analytic for $\Re(s)>1/(2\alpha_{\cW})$.

The final assertion follows just as in the proof of Lemma \ref{L:tech2}.
\end{proof}

\begin{definition} \label{D:numbers}
For each positive integer $n$ and each character $\chi$ of
$\Cl_{\cM}'(\Lambda)$, set
\begin{equation*}
d_n(\chi) := |\{ t \in T'| \text{$\chi_t = \1$ and $\cW(t) =n$} \}|.
\end{equation*}
Hence
\begin{align*}
d_n(\1) &= |\{ t \in T'| \text{$\cW(t) =n$} \}| \\
&= \max_{\chi} \{ d_n(\chi) \}.
\end{align*}
Write
\begin{equation*}
b_n(\chi) := \lim_{s \to \frac{1}{n}} \left( s - \frac{1}{n}
\right)^{d_n(\1)} D(\chi,s).
\end{equation*}
\qed
\end{definition}

\begin{proposition} \label{P:poles}
Let $1 \leq n < 2\alpha_{\cW}$ be a positive integer.
\smallskip

(a) The function $D(\1, s)$ has a pole of exact order $d_n(\1)$ at
  $s=1/n$.
\smallskip

(b) If $\chi \neq \1$, then $D(\chi,s)$ has a pole of order at most
$d_n(\chi)$ at $s=1/n$.
\smallskip

(c) For each $\fc \in \Cl_{\cM}'(\Lambda)$, the function
$D_{\fc,\cM}(s)$ has a pole of order at most $d_n(\1)$ at $s=1/n$, and 
\begin{equation*}
\lim_{s \to \frac{1}{n}} \left( s - \frac{1}{n} \right)^{d_n(\1)} D_{\fc,\cM}(s)
= \frac{1}{|\Cl'_{\cM}(\Lambda)|} \sum_{\chi} \overline{\chi}(\fc)
b_n(\chi).
\end{equation*}

These are the only poles of the functions $D(\chi,s)$ and
$D_{\fc,\cM}(s)$ in the half-plane $\Re(s)>1/(2\alpha_{\cW})$.
\end{proposition}

\begin{proof}
From \eqref{E:liso} and Proposition \ref{P:comp2}, we have
\begin{equation} \label{E:dprod}
D(s,\chi) = L_{\Lambda}(s, \chi) \cdot \psi(s,\chi) = \left[ \prod_{t
    \in T'} L_t(s,\chi_t) \right] \cdot \psi(s,\chi),
\end{equation}
where $\psi(s, \chi)$ is analytic for $\Re(s) > 1/(2\alpha_{\cW})$. For
each $t \in T'$, the Dirichlet $L$-function $L_t(s,\chi_t)$ is entire
unless $\chi_t = \1_t$ in which case it has a single (simple) pole at
$s = 1/\cW(t)$. This implies that, for \textit{any} positive integer
$n$, the function $L_{\Lambda}(s,\chi)$ has a pole of order exactly
$d_n(\chi)$ (which of course may be equal to zero!) at $s=1/n$.

If $1 \leq n < 2 \alpha_{\cW}$, then it follows from \eqref{E:dprod}
that $D(s,\chi)$ has a pole of exact order $d_n(\chi)$ at $s=1/n$, unless
$\psi(1/n, \chi)=0$, in which case the pole might be of lower order.
We note that each Euler factor $D_p(1/n, \1)$ is non-zero because it
is a finite sum of positive terms. Hence Proposition \ref{P:comp2}
implies that $\psi(1/n, \1) \neq 0$, and so $D(s, \1)$ has a pole of
order exactly $d_n(\1)$ at $s = 1/n$, as claimed. This proves parts (a) and (b).

Part (c) follows immediately from \eqref{E:dfourier}. The final
assertion of the Proposition is a direct consequence of
\eqref{E:dprod}, \eqref{E:dfourier} and Proposition \ref{P:comp2}.
\end{proof}

\begin{lemma} \label{L:ind} For any positive integer $n$ with $1
  \leq n \leq 2\alpha_{\cW}$, the number
\begin{equation*}
\lim_{s \to \frac{1}{n}} \left( s - \frac{1}{n} \right)^{d_n(\1)} D_{\fc,
  \cM}(s)
\end{equation*}
is independent of $\fc$ if and only if $b_n(\chi)=0$ for all $\chi
\neq \1$.
\end{lemma}

\begin{proof} This follows directly from Proposition \ref{P:poles}(c),
  via linear independence of characters.
\end{proof}

We can now state a necesary and sufficient condition for
$\kappa_{\cW}(\fc,X;\cM)$ to be asymptotically independent of $\fc$.

\begin{proposition} \label{P:ind}
We have that $\kappa_{\cW}(\fc,X;\cM)$ is asymptotically
independent of $\fc$ if and only if $b_{\alpha_{\cW}}(\chi) = 0$ for all
$\chi \neq \1$.
\end{proposition}

\begin{proof} This follows directly from Lemma \ref{L:ind} and
Definition \ref{D:ind}. We first note that Proposition
\ref{P:poles}(a) implies that $b_{\alpha_{\cW}}(\1)$ is always
strictly greater than zero. If $b_{\alpha_{\cW}}(\chi) = 0$ for all
$\chi \neq \1$, then it is easy to see that the numbers
$\tau(\fc;\cM), \beta(\fc;\cM)$ and $\delta(\fc;\cM)$ are independent of
$\fc$, which in turn implies that $\kappa_{\cW}(\fc,X;\cM)$ is
asymptotically independent of $\fc$.

On the other hand, if $b_{\alpha_{\cW}}(\chi) \neq 0$ for some $\chi
\neq \1$, then Proposition \ref{P:poles}(c) implies (via linear
independence of characters) that $\tau(\fc;\cM)$ is not independent of
$\fc$, and so we deduce that $\kappa_{\cW}(\fc,X;\cM)$ cannot be
asymptotically independent of $\fc$.
\end{proof}

\begin{corollary} \label{C:ind}
(a) If $\kappa_{\cW}(\fc, X;\cM)$ is asymptotically independent of
  $\fc$, then for each $\fc \in \Cl_{\cM}'(\Lambda)$ we have that
  $\beta(\fc;\cM) = 1/\alpha_{\cW}$. Also, $D_{\fc, \cM}(s)$ has a pole of
  exact order $d_{\alpha_{\cW}}(\1)$ at $s = 1/\alpha_{\cW}$, and
\begin{equation*}
\lim_{s \to \frac{1}{\alpha_{\cW}}} \left( s - \frac{1}{\alpha_{\cW}}
\right)^{d_{\alpha_{\cW}(\1)}} D_{\fc, \cM}(s) = b_{\alpha_{\cW}}(\1).
\end{equation*}
\smallskip

(b) If $\cW$ is constant on $T'$ (so $\cW(t)= \alpha_{\cW}$ for all $t
\in T'$), then $\kappa_{\cW}(\fc,X;\cM)$ is
asymptotically independent of $\fc$, and $d_{\cW}(\1)= |T'|$. We have
\begin{equation} \label{E:kappacon}
\kappa_{\cW}(\fc,X;\cM) \sim \frac{\tau(\cM)
  \alpha_{\cW}}{\Gamma(|T'|)} \cdot X^{1/\alpha_{\cW}} \cdot (\log
X)^{|T'|-1}
\end{equation}
as $X \to \infty$, where here we have written $\tau(\cM)$ rather than
$\tau(\fc;\cM)$ as this term is independent of $\fc$. 
\end{corollary}

\begin{proof} 
This follows readily from the definitions, together with Proposition \ref{P:kappa}.
\end{proof}

The following result gives an example of a situation in which
$\kappa_{\cW}(\fc,X;\cM)$ is not asymptotically independent of
$\fc$.

\begin{proposition} \label{P:kappafail}
Suppose that $K \Lambda$ is totally split over $K$ (i.e. in the
Wedderburn decomposition \eqref{E:wiso} of $K\Lambda$, we have
$K(t)=K$ for all $t \in T$), and that the weight $\cW$ is not constant
on $T'$. Then $\kappa_{\cW}(\fc,X;\cM)$ is not asymptotically
independent of $\fc$.
\end{proposition}

\begin{proof}

It suffices to show that if $K \Lambda$ is totally split over $K$ and
$\cW$ is not constant on $T'$, then there exists a non-trivial
character $\chi$ of $\Cl_{\cM}'(\Lambda)$ with $d_{\alpha_{\cW}}(\chi)
= d_{\alpha_{\cW}}(\1)$ such that $D(s,\chi)$ has a pole of exact
order $d_{\alpha_{\cW}}(\1)$ at $s = 1/\alpha_{\cW}$ (see Proposition
\ref{P:ind}). We see immediately from Proposition \ref{P:comp2} that
to do this, it suffices to exhibit a non-trivial $\chi$ satisfying
$d_{\alpha_{\cW}}(\chi) = d_{\alpha_{\cW}}(\1)$ and
$D_p(1/\alpha_{\cW},\chi) \neq 0$ for all rational primes $p$.

Suppose that $\fp$ is a prime of $O_K$. Since $K\Lambda$ is totally
split, it follows from Proposition \ref{P:fideals} that the set of
ideals of $\fF \subseteq I(\Lambda)$ lying above $\fp$ consists
precisely of all ideals of the form $\fa(\sigma) = (\fa(\sigma)_t)_{t
  \in T}$ for each $\sigma \in T'$, where $\fa(\sigma)_t = O_K$ if $t
\in T$ with $t \neq \sigma$, and $\fa(\sigma)_{\sigma}= \fp$. For each
character $\chi = (\chi)_{t \in T}$ of $\Cl_{\cM}'(\Lambda)$, set
\begin{equation*}
D_{\fp}(s, \chi):= 1 + \sum_{\sigma \in T'}
\chi_{\sigma}(\fp)[O_K:\fp]^{-\cW(\sigma)s}.
\end{equation*}
Then it is not hard to check that (see \eqref{E:Deuler})
\begin{equation} \label{E:Dprod}
D_p(s,\chi) = \prod_{\fp \mid p} D_{\fp}(s,\chi).
\end{equation}

We now observe that, as $\cW$ is not constant on $T'$, we may choose
$t_0 \in T'$ such that $\cW(t_0) > \alpha_{\cW}$. Let $S(t_0)$ denote
the set of all characters $\chi$ of $\Cl_{\cM}'(\Lambda)$ such that
$\chi_t = \1$ for all $t \neq t_0$. Plainly $|S(t_0)| =
|\Cl_{\cM}(O_K)|>1$, where $\C_{\cM}(O_K)$ denotes the ray class group
modulo $\cM$ of $O_K$, and we have $d_{\alpha_{\cW}}(\chi) =
d_{\alpha_{\cW}}(\1)$ for all $\chi \in S(t_0)$. Now, for any $\chi
\in S(t_0)$, we have
\begin{equation} \label{E:notzero}
D_{\fp}(1/\alpha_{\cW},\chi) = 1 + \left(\sum_{\substack{\sigma \in
    T'\\ \sigma \neq t_0}} [O_K:\fp]^{-\cW(\sigma)/\alpha_{\cW}}
\right) + \chi_{t_0}(\fp) [O_k:\fp]^{-\cW(t_0)/\alpha_{\cW}}.
\end{equation}
Since $|\chi_{t_0}(\fp) [O_k:\fp]^{-\cW(t_0)/\alpha_{\cW}}| < 1$, it
follows from \eqref{E:notzero} that $|D_{\fp}(1/\alpha_{\cW}, \chi)| >
0$, and so we deduce from \eqref{E:Dprod} that
$D_p(1/\alpha_{\cW},\chi) \neq 0$ also, as required.
\end{proof}

%%%%%%%%%%%%%%%%%%%%%%%%%%%%%%%%%%%%%%%%%%%%%%%%%%%%%%%%%%%%%%%%%%%%%%%%%%%%%%%%%%%%%%%%%%%%%%%%%%%

\section{An equidistribution result} \label{S:equi}

Let $c \in \cR(O_KG)$ be a realisable class. In this section we shall
discuss the number $N_{\cW}(c,X;\cM)$ of tame Galois $G$-extensions
$K_h/K$ for which $(O_h) = c$, $(D_{\cW}(K_h/K), \cM)=1$ and
$D_{\cW}(K_h/K) \leq X$, under the assumption that
$\kappa_{\cW}(\fc,X;\cM)$ is asymptotically independent of $\fc$.

Suppose therefore that $\kappa_{\cW}(\fc,X;\cM)$ is asymptotically
independent of $\fc$. Recall (see Definition \ref{D:psi}) that we have
a homomorphism
\begin{equation*}
\psi: H^{1}_{tr}(K, G) \to \cC(O_KG)
\end{equation*}
with finite kernel, and a surjective homomorphism (see Proposition
\ref{P:finiteness})
\begin{equation*}
f_{\cM}: \Cl'_{\cM}(\Lambda) \to \frac{J(K\Lambda)}{\cP_{\theta}}.
\end{equation*}
 
\begin{theorem} \label{T:main}
With the above hypotheses and notation, we have
\begin{align*}
N_{\cW}(c,X;\cM) &= |\Ker(\psi)| \cdot |\Ker(f_{\cM})| \cdot
\kappa_{\cW}(\fc,X;\cM) \\
&\sim \frac{\alpha_{\cW} \cdot |\Ker(\psi)| \cdot
  |\Ker(f_{\cM})| \cdot b_{\alpha_{\cW}}(\1)}{\Gamma(d_{\alpha_{\cW}}(\1))} \cdot
X^{\frac{1}{\alpha_{\cW}}} \cdot (\log X)^{d_{\cW}(\1)-1}
\end{align*}
as $X \to \infty$.
\end{theorem}

\begin{proof}
This follows directly from \eqref{E:countfunc}, Proposition
\ref{P:kappa} and Corollary \ref{C:ind}.
\end{proof}

We thus see that if $\kappa_{\cW}(\fc,X;\cM)$ is asymptotically
independent of $\fc$, then the tame Galois $G$-extensions $K_h$ of $K$
with $D_{\cW}(K_h/K)$ coprime to $\cM$ are equidistributed amongst the
realisable classes in $\Cl(O_KG)$ as $X \to \infty$.
\smallskip

\begin{example} \label{E:foster}
Let us now consider the case treated by K. Foster in \cite{Fo}.
Let $l$ be a prime, and suppose that $G$ is an elementary abelian
$l$-group of order $l^k$. Suppose also that $\cW= \cW_{\disc}$ (see
Example \ref{E:exweight}(1)). For each $t \in T'$, we have
\begin{equation*}
\cW(t) = \frac{(|t|-1)|G|}{|t|} = \frac{(l-1)l^k}{l} = l^{k-1}(l-1) =
\phi(|G|),
\end{equation*}
where $\phi$ denotes the Euler $\phi$-function. Hence $\cW$ is
constant on $T'$, and so Corollary \ref{C:ind}(b) implies that
$\kappa(\fc,X;\cM)$ is asymptotically independent of $\fc$. If we take
\begin{equation*}
\cM = |G|^2 \Lambda = l^{2k} \Lambda,
\end{equation*}
then for each $c \in \cR(O_KG)$, we have $N_{\cW}(c,X;\cM) =
N_{\disc}(c,X)$ because, since $G$ is an $l$-group, a
$G$-extension $K_h/K$ is tamely ramified if and only if it is
unramified at all primes dividing $l$.

We have that $\alpha_{\cW}=1/\phi(|G|)$, and $d_{\cW}(\1)=
|T'|$. Theorem \ref{T:main} and Corollary \ref{C:ind} therefore imply
that
\begin{equation} \label{E:kurt1}
N_{\cW}(c,X) \sim \frac{\phi(|G|) \cdot |\Ker(\psi)| \cdot |\Ker(f_{\cM})| \cdot
  b_{\alpha_{\cW}}(\1)}{\Gamma(|T'|)} \cdot X^{1/\phi(|G|)} \cdot
(\log(X))^{|T'|-1}.
\end{equation}
The tower law for discriminants implies that for each tamely ramified
$G$-extension $K_h/K$ we have 
\begin{equation*}
\disc(K_h/\bQ) = D_{\cW}(K_h/K) \disc(K/\bQ)^{|G|}
\end{equation*}
and this in turn implies that
\begin{equation} \label{E:ndisc}
N_{\disc}(c,X) = N_{\cW}(c, X/\disc(K/\bQ)^{|G|}).
\end{equation}

From \eqref{E:kurt1} and \eqref{E:ndisc}, we have
\begin{align*}
N_{\disc}(c,X) &\sim \frac{\phi(|G|)\cdot |\Ker(\psi)| \cdot |\Ker(f_{\cM})| \cdot
  b_{\alpha_{\cW}}(\1)}{\Gamma(|T'|)} 
\cdot \left( \frac{X}{\disc(K/\bQ)^{|G|}}\right)^{1/\phi(|G|)} \\ 
&\times \left( \log \left(\frac{X}{\disc(K/\bQ)^{|G|}}\right) \right)^{|T'|-1} \\
&= \frac{|\Ker(f_{\cM})| \cdot |\Ker(\psi)| \cdot b_{\alpha_{\cW}}(\1)}{\Gamma(|T'|)}
  \cdot Y \cdot (\log(Y))^{|T'|-1},
\end{align*}
where $Y^{\phi(|G|)} \cdot \disc(K/\bQ)^{|G|} = X$.

Theorem \ref{T:foster} of the Introduction now follows immediately.
\qed
\end{example}

\begin{example} \label{E:bisi}
Suppose now that $G$ is any finite abelian group. Let $\cW=\cW_{\ram}$
(see Example \ref{E:exweight}(2)), and set $\cM = |G|^2
\Lambda$. Then, for each $c \in \cR(O_KG)$, it follows from the
definitions that $N_{\cD}(c,X)$ (see Theorem \ref{T:main} of the
Introduction) is equal to $N_{\cW}(c,X;\cM)$.

As $\cW$ is constant on $T'$, Corollary \ref{C:ind}(b) implies that
$\kappa_{\cW}(\fc,X;\cM)$ is asymptotically independent of $\fc$. It
is not hard to check that $\alpha_{\cW}=1$ and
$d_{\cW}(\1)=|T'|$. Theorem \ref{T:main} now implies that
\begin{equation*}
N_{\cW}(c,X;\cM) \sim \frac{|\Ker(\psi)| \cdot |\Ker(f_{\cM})| \cdot
  b_{\alpha_{\cW}}(\1)}{\Gamma(|T'|)} \cdot X \cdot (\log X)^{|T'|-1}.
\end{equation*}

This implies Theorem \ref{T:B} of the Introduction. \qed
\end{example}

%%%%%%%%%%%%%%%%%%%%%%%%%%%%%%%%%%%%%%%%%%%%%%%%%%%%%%%%%%%%%%%%%%%%%%%%%%%%%%%%%%%%%%%%%%%%%%%%%%%%%%%%

\section{Field Extensions}  \label{S:field}

In this section we shall show that, in a large number of cases, it
makes no difference if we work with tame Galois field extensions of
$K$ with group $G$, rather than tame Galois $G$-algebra extensions of
$K$ (see Proposition \ref{P:allsame} below). We shall do this via a
modification of a technique described by Foster in \cite[Corollary 1.6
  and Lemma 4.15]{Fo}.

Recall that, via the Wedderburn decomposition \eqref{E:wiso} of
$\Lambda$, each ideal $\fa$ in $I(\Lambda)$ may be written $\fa =
(\fa_t)_{t \in T}$, where each $\fa_t$ is a fractional ideal of
$O_{K(t)}$.

\begin{definition} \label{E:kappat}
For each coset $\fc$ of $\cP_{\cM}$ in $J(K \Lambda)$ and each $t \in
T'$, set
\begin{equation}
\kappa^{(t)}_{\cW}(\fc,X;\cM) := |\left\{ f \in \bF \cap \fc \mid \text{
  $(\co(f), \cM)=1, \co(f)_t=O_{K(t)}$ and $[\Lambda:
  d_{\cW}(\co(f))] \leq X$} \right\}|
\end{equation}

Define
\begin{equation} \label{kappafull}
\kappa_{\cW}^{\full}(\fc,X;\cM):= \kappa_{\cW}(\fc,X;\cM) - \sum_{t
  \in T'} \kappa^{(t)}_{\cW}(\fc,X;\cM).
\end{equation}
We see that $\kappa_{\cW}^{\full}(\fc,X;\cM)$ is equal to the number of ideles
$f \in \bF$ such that $(\co(f), \cM)=1$, $[\Lambda: d_{\cW}(\co(f))]
\leq X$, and $\co(f)_t \neq O_{K(t)}$ for all $t \in T'$. \qed
\end{definition}

\begin{proposition} \label{P:kappafull}
Suppose that $\cW$ is constant of $T'$, so $\cW(t) = \alpha_{\cW}$ for
all $t \in T'$.

Then for each $t \in T'$, we have
\begin{equation} \label{E:gapzero}
\lim_{X \to \infty}
\frac{\kappa^{(t)}_{\cW}(\fc,X;\cM)}{\kappa_{\cW}(\fc,X;\cM)} = 0,
\end{equation}
and so 
\begin{equation*}
\kappa_{\cW}(\fc,X;\cM) \sim \kappa_{\cW}^{\full}(\fc,X;\cM)
\end{equation*}
as $X \to \infty$.
\end{proposition}

\begin{proof} The second assertion is an immediate consequence of the
  first, and so we shall just explain how to prove \eqref{E:gapzero}. 

Let $\Lambda^{(t)}$ denote the algebra $\Lambda$ with the Wedderburn
component corresponding to $t$ deleted. Then, carrying out all of the
arguments of Sections \ref{S:euler}--\ref{S:poles} with $\Lambda$
replaced by $\Lambda^{(t)}$, we see from the variant of Corollary
\ref{C:ind}(b) that we obtain in this way that
\begin{equation*}
\kappa^{(t)}_{\cW}(\fc,X;\cM) \sim \frac{\tau_1(\cM)
  \alpha_{\cW}}{\Gamma(|T'|)} \cdot X^{1/\alpha_{\cW}} \cdot (\log
X)^{|T'|-2}
\end{equation*}
Since, from the original version of Corollary \ref{C:ind}(b), we have that
\begin{equation*}
\kappa_{\cW}(\fc,X;\cM) \sim \frac{\tau(\cM)
  \alpha_{\cW}}{\Gamma(|T'|)} \cdot X^{1/\alpha_{\cW}} \cdot (\log
X)^{|T'|-1},
\end{equation*}
the equality \eqref{E:gapzero} follows at once.
\end{proof}

\begin{remark} Proposition \ref{P:kappafull} does not necessarily hold if $\cW$ is not
  constant on $T'$. \qed
\end{remark}

\begin{proposition} \label{P:leonfield}(McCulloh)
Suppose that $h \in H^{1}_{tr}(K,G)$ with $(O_h) = c \in
\cR(O_KG)$. Recall that there exists a unique $f \in \bF$ such that
$\rho(c) = \psi(h)^{-1} \theta(f)$ (see Remark \ref{R:newleon}(1)). If
$\co(f)_t \neq O_{K(t)}$ for all $t \in T'$, then $K_h$ is a field.
\end{proposition}

\begin{proof} See \cite[proof of Theorem 6.7(a)--(b)]{Mc}. The
  essential idea is as follows. One first shows that if $K_h$ is not a
  field, then it contains a Galois subalgebra extension $H/K$ with $K
  \neq H$ and $H/K$ unramified. One then establishes that, on the other
  hand, if $\co(f)_t \neq O_{K(t)}$ for all $t \in T'$, then every
  Galois subalgebra extension $H/K$ of $K_h$ with $H \neq K$ is in
  fact ramified. Hence, if $\co(f)_t \neq O_{K(t)}$ for all $t \in
  T'$, then it follows that $K_h$ must be a field. 
\end{proof}

For each $c \in \cR(O_KG)$, and each real number $X>0$, write
$N_{\cW}^{f}(c,X;\cM)$ for the number of tame Galois $G$-extensions
$K_h/K$ for which $(O_h)=c$, $D_{\cW}(K_h/K)$ is coprime to $\cM$,
$D_{\cW}(K_h/K) \leq X$, and $K_h$ is a field.

\begin{proposition} \label{P:allsame}
Suppose that $\cW$ is constant on $T'$. Then, for each $c \in
\cR(O_KG)$, we have
\begin{equation} \label{E:allsame}
N_{\cW}^{f}(c,X;\cM) \sim N_{\cW}(c,X;\cM)
\end{equation}
as $X \to \infty$.
\end{proposition}

\begin{proof}
If $\cW$ is constant on $T'$, then $\kappa_{\cW}(\fc,X;\cM)$ is
asymptotically independent of $\fc$. Hence we have that (see Theorem
\ref{T:main}) 
\begin{equation} \label{E:alcount}
N_{\cW}(c,X;\cM) = |\Ker(\psi)| \cdot |\Ker(f_{\cM})| \cdot
\kappa_{\cW}(\fc,X;\cM)
\end{equation}
for any $\fc \in J(K\Lambda)/\cP_{\cM}$.

Proposition \ref{P:leonfield} implies that
\begin{equation} \label{E:fcount}
|\Ker(\psi)| \cdot |\Ker(f_{\cM})| \cdot
\kappa_{\cW}^{\full}(\fc,X;\cM) \leq N_{\cW}^{f}(c,X;\cM) \leq
N_{\cW}(c,X;\cM).
\end{equation}
Now from \eqref{E:alcount}, we see that
\begin{equation*}
\lim_{X \to \infty} \frac{|\Ker(\psi)| \cdot |\Ker(f_{\cM})| \cdot
\kappa_{\cW}^{\full}(\fc,X;\cM)}{N_{\cW}(c,X;\cM)} =
\lim_{X \to \infty}
\frac{\kappa_{\cW}^{\full}(\fc,X;\cM)}{\kappa_{\cW}(\fc,X;\cM)} =1.
\end{equation*}
where the second equality follows from Proposition
\ref{P:kappafull}. Hence \eqref{E:fcount} implies that
\begin{equation*}
\lim_{X \to \infty} \frac{N_{\cW}^{f}(c,X;\cM)}{N_{\cW}(c,X;\cM)} = 1
\end{equation*}
also.
\end{proof}

%%%%%%%%%%%%%%%%%%%%%%%%%%%%%%%%%%%%%%%%%%%%%%%%%%%%%%%%%%%%%%%%%%%%%%%%%%%%%%%%%%%%%%%%%%%%%%%%%%%%%%

\section{Futher questions} \label{S:sigh}

Theorem \ref{T:main} implies that if $\kappa_{\cW}(\fc,X;\cM)$ is
asymptotically independent of $\fc$, then the second part of Question
\ref{Q:general} has an affirmative answer, i.e. the limit
\begin{equation*}
Z_{\cW}(c;\cM):= \lim_{X \to \infty}
\frac{N_{\cW}(c,X;\cM)}{M_{\cW}(X)}
\end{equation*}
is independent of $c \in \cR(O_KG)$. What happens if the assumption
that $\kappa_{\cW}(\fc,X;\cM)$ is a asymptotically independent of
$\fc$ is dropped? We see from \eqref{E:countfunc} that if $c_1, c_2
\in \cR(O_KG)$, then
\begin{equation*}
N_{\cW}(c_1,X;\cM) \sim N_{\cW}(c_2,X;\cM)
\end{equation*}
as $X \to \infty$ if and only if
\begin{equation} \label{E:noway}
\sum_{\fc \in f_{\cM}^{-1}(c_1)} \kappa_{\cW}(\fc,X;\cM) \sim 
\sum_{\fc \in f_{\cM}^{-1}(c_2)} \kappa_{\cW}(\fc,X;\cM)
\end{equation}
as $X \to \infty$.

If $\kappa_{\cW}(\fc,X;\cM)$ is not asymptotically independent of
$\fc$, then it seems unreasonable to expct \eqref{E:noway} to hold
for all $c_1,c_2 \in \cR(O_KG)$. In this case, it is therefore
probably no longer true in general that $Z_{\cW}(c;\cM)$ is
independent of $c$, and one would expect the behaviour of $Z_{\cW}(c;\cM)$
with respect to $c$ to depend very much upon the choice of $\cW$. At
present we have no results or examples in this situation. In order to
produce an explicit example in which $Z_{\cW}(c;\cM)$ depends upon
$c$, there are two main issues that need to be addressed.

Suppose that $G$ is a finite abelian group which is such that
$\cR(O_KG) \neq 0$. (It is possible to produce such examples for many
different $K$ using work of Brinkhuis \cite{Br}.) One would first
have to show that, in the example under consideration,
$\kappa_{\cW}(\fc,X;\cM)$ is not independent of $\fc$. This can be
done in many cases by appealing to Proposition \ref{P:kappafail} above.
One would then have to show that \eqref{E:noway} fails for some $c_1,
c_2 \in \cR(O_KG)$. The point here is that this is not directly
implied by $\kappa_{\cW}(\fc,X;\cM)$ being asymptotically dependent
upon $\fc$: one has to rule out the (admittedly unlikely) possibility of
the $\kappa_{\cW}(\fc,X;\cM)$ varying with $\fc$ in such a way that
\eqref{E:noway} always holds.

One possible approach towards dealing with these issues would be
to try and work with $L$-functions constructed directly from
$J(K\Lambda)/\cP_{\theta}$ directly, avoiding the use of the group
$\Cl_{\cM}'(\Lambda)$ entirely (cf. \cite{BR}, for example). An
additional advantage of such an approach is that it would also
presumably allow us to consider $G$-extensions $K_h/K$ in which the
places dividing $|G|$ are allowed to be tamely ramified.
Finally, we remark also that it should be possible to use the methods of this
paper to investigate similar questions in the setting of function
fields (see \cite{AB}). We hope to return to these topics in future
work.

%\newpage

\end{document}